\documentclass{birkmult}
\usepackage{amssymb,amsmath}
 \newtheorem{theorem}{Theorem}[section]
 \newtheorem{corollary}[theorem]{Corollary}
 \newtheorem{lemma}[theorem]{Lemma}
 \newtheorem{proposition}[theorem]{Proposition}
 \theoremstyle{definition}
 \newtheorem{definition}[theorem]{Definition}
 \theoremstyle{remark}
 \newtheorem{remark}[theorem]{Remark}
 \newtheorem{example}{Example}
 \newtheorem*{problem}{Problem}
 \numberwithin{equation}{section}

\begin{document}
\title[Constrained energy problems with external fields]
 {Constrained energy problems with external fields\\ for infinite dimensional vector measures}
\author[N.~Zorii]{Natalia Zorii}

\address{%
Institute of Mathematics\\
National Academy of Sciences of Ukraine\\
3 Tereshchenkivska Str.\\
01601, Kyiv-4\\
Ukraine}

\email{natalia.zorii@gmail.com}



\date{December 31, 2009}

\begin{abstract}
We consider a constrained minimal energy problem with an external
field over noncompact classes of infinite dimensional vector
measures $(\mu^i)_{i\in I}$ on a locally compact space. The
components $\mu^i$ are positive measures with the properties $\int
g_i\,d\mu^i=a_i$ and $\sigma^i-\mu^i\geqslant0$ (where $a_i$, $g_i$,
and~$\sigma^i$ are given) and supported by closed sets~$A_i$ with
the sign~$+1$ or~$-1$ prescribed such that $A_i\cap A_j=\varnothing$
whenever ${\rm sign}\,A_i\ne{\rm sign}\,A_j$, and the law of
interaction of $\mu^i$, $i\in I$, is determined by the matrix
$\bigl({\rm sign}\,A_i\,{\rm sign}\,A_j\bigr)_{i,j\in I}$. For all
positive definite kernels satisfying Fuglede's condition of
consistency between the vague (=\,weak$^*$) and strong topologies,
sufficient conditions for the existence of minimizers are
established and their uniqueness and vague compactness are studied.
Examples illustrating the sharpness of the sufficient conditions are
provided. We also analyze continuity properties of minimizers in the
vague and strong topologies when $A_i$ and~$\sigma^i$ are varied
simultaneously. The results are new even for classical kernels in
$\mathbb R^n$, which is important in applications.
\end{abstract}

\maketitle

\section{Introduction}

In all that follows, $\mathrm X$ denotes a locally compact Hausdorff
space and $\mathfrak M=\mathfrak M(\mathrm X)$ the linear space of
all real-valued scalar Radon measures~$\nu$ on~$\mathrm X$ equipped
with the {\em vague\/} \mbox{($={}${\em weak}$^*$)} topology, i.e.,
the topology of pointwise convergence on the class $\mathrm
C_0(\mathrm X)$ of all real-valued continuous functions~$\varphi$
on~$\mathrm X$ with compact support.

A {\em kernel\/}~$\kappa$ on $\mathrm X$ is meant to be an element
from $\mathrm\Phi(\mathrm X\times\mathrm X)$, where
$\mathrm\Phi(\mathrm Y)$ consists of all lower semicontinuous
functions~$\psi:\mathrm Y\to(-\infty,\infty]$ such that
$\psi\geqslant0$ unless $\mathrm Y$ is compact. Given
$\nu,\nu_1\in\mathfrak M$, the {\em mutual energy\/} and the {\it
potential\/} relative to the kernel~$\kappa$ are defined by
\[
\kappa(\nu,\nu_1):=\int\kappa(x,y)\,d(\nu\otimes\nu_1)(x,y)\quad\text{and}\quad
\kappa(\,\cdot\,,\nu):=\int\kappa(\,\cdot\,,y)\,d\nu(y),
\]
respectively. (When introducing notation, we always tacitly assume
the corresponding object on the right to be well defined --- as a
finite number or $\pm\infty$.)

For $\nu=\nu_1$ the mutual energy $\kappa(\nu,\nu_1)$ defines the
{\em energy\/} $\kappa(\nu,\nu)$ of~$\nu$. We denote by $\mathcal
E=\mathcal E_\kappa(\mathrm X)$ the set of all $\nu\in\mathfrak M$
with $-\infty<\kappa(\nu,\nu)<\infty$.

We shall mainly be concerned with a {\em positive definite\/}
kernel~$\kappa$, which means that it is symmetric (i.e.,
$\kappa(x,y)=\kappa(y,x)$ for all $x,y\in\mathrm X$) and the energy
$\kappa(\nu,\nu)$, $\nu\in\mathfrak M$, is nonnegative whenever
defined. Then $\mathcal E$ forms a pre-Hil\-bert space with the
scalar product $\kappa(\nu,\nu_1)$ and the seminorm
$\|\nu\|_\mathcal E:=\|\nu\|_\kappa:=\sqrt{\kappa(\nu,\nu)}$
(see~\cite{F1}); the topology on~$\mathcal E$, determined by this
seminorm, is called {\em strong}. A positive definite
kernel~$\kappa$ is {\em strictly positive definite\/} if the
seminorm $\|\cdot\|_\mathcal E$ is a norm.

Given a closed set $F\subset\mathrm X$, we denote by $\mathfrak
M^+(F)$ the convex cone of all nonnegative $\nu\in\mathfrak M$
supported by~$F$, and let $\mathcal E^+(F):=\mathfrak
M^+(F)\cap\mathcal E$. Also write $\mathfrak M^+:=\mathfrak
M^+(\mathrm X)$ and $\mathcal E^+:=\mathcal E^+(\mathrm X)$.

We consider a countable, locally finite collection
$\mathbf{A}=(A_i)_{i\in I}$ of fixed closed sets $A_i\subset\mathrm
X$ with the sign~$+1$ or $-1$ prescribed such that the oppositely
signed sets are mutually disjoint. Let $\mathfrak M^+(\mathbf{A})$
stand for the Cartesian product $\prod_{i\in I}\,\mathfrak
M^+(A_i)$; then an element $\boldsymbol{\mu}$ of $\mathfrak
M^+(\mathbf{A})$ is a (nonnegative) {\em vector measure\/}
$(\mu^i)_{i\in I}$ with the components $\mu^i\in\mathfrak M^+(A_i)$.
The topology of the product space $\prod_{i\in I}\,\mathfrak
M^+(A_i)$, where every $\mathfrak M^+(A_i)$ is endowed with the
vague topology, is likewise called {\em vague\/}.

If a vector measure $\boldsymbol{\mu}\in\mathfrak M^+(\mathbf{A})$
and a vector-valued function $\mathbf{u}=(u_i)_{i\in I}$ with
$\mu^i$-measurable components $u_i:A_i\to[-\infty,\infty]$ are
given, then for brevity we write\footnote{Here and in the sequel, an
expression $\sum_{i\in I}\,c_i$ is meant to be well defined provided
so is every summand~$c_i$ and the sum does not depend on the order
of summation
--- though might be $\pm\infty$.  Then, by Riemann series theorem,
the sum is finite if and only if the series converges absolutely.}
\begin{equation*}\langle\mathbf{u},\boldsymbol{\mu}\rangle:=\sum_{i\in I}\,\int
u_i\,d\mu^i.\label{notation}\end{equation*}

Let a kernel~$\kappa$ be fixed. In accordance with an electrostatic
interpretation of a condenser, we assume that the interaction
between the charges lying on the conductors~$A_i$, $i\in I$, is
characterized by the matrix $(\alpha_i\alpha_j)_{i,j\in I}$, where
$\alpha_i:={\rm sign}\,A_i$. Then the {\em energy\/} of
$\boldsymbol{\mu}\in\mathfrak M^+(\mathbf{A})$ is defined by the
formula
\begin{equation*}\label{vectorenergy}\kappa(\boldsymbol{\mu},\boldsymbol{\mu}):=\sum_{i,j\in
I}\,\alpha_i\alpha_j\kappa(\mu^i,\mu^j).\end{equation*} We denote by
$\mathcal E^+(\mathbf{A})$ the set of all
$\boldsymbol{\mu}\in\mathfrak M^+(\mathbf{A})$ with
$-\infty<\kappa(\boldsymbol{\mu},\boldsymbol{\mu})<\infty$.

Also fix a vector-valued function $\mathbf{f}=(f_i)_{i\in I}$,
treated as an {\em external field\/}, and assume it to satisfy one
of the following two cases:\medskip

\begin{tabular}{rl}
{\em Case\/} I.& $f_i\in\mathrm\Phi(\mathrm X)$ {\em for all\/} $i\in I$;\\[2pt] {\it
Case\/} II.& $f_i=\alpha_i\kappa(\cdot,\zeta)$ {\em for all $i\in
I$, where
$\zeta\in\mathcal E$ is a signed measure\/}.\\
\end{tabular}\medskip

\noindent Furthermore, suppose each~$f_i$ to affect the charges
on~$A_i$ only; then the $\mathbf{f}$-{\em weighted energy\/} of
$\boldsymbol{\mu}\in\mathfrak M^+(\mathbf{A})$ is given by the
expression
\begin{equation}\label{wen}G_{\mathbf{f}}(\boldsymbol{\mu}):=
\kappa(\boldsymbol{\mu},\boldsymbol{\mu})+2\langle\mathbf{f},\boldsymbol{\mu}\rangle.\end{equation}
Let $\mathcal E_{\mathbf f}^+(\mathbf{A})$ consist of all
$\boldsymbol{\mu}\in\mathcal E^+(\mathbf{A})$ with
$-\infty<G_{\mathbf{f}}(\boldsymbol{\mu})<\infty$.

Also fix a vector measure $\boldsymbol{\sigma}\in\mathfrak
M^+(\mathbf{A})$, serving as a {\em constraint\/}, a numerical
vector $\mathbf{a}=(a_i)_{i\in I}$ with $a_i>0$ for all~$i\in I$,
and a vector-valued function $\mathbf{g}=(g_i)_{i\in I}$, where all
the $g_i:A_i\to(0,\infty)$ are continuous. In the study, we are
interested in the problem of minimizing
$G_{\mathbf{f}}(\boldsymbol{\mu})$ over the class of all
$\boldsymbol{\mu}\in\mathcal E_{\mathbf f}^+(\mathbf{A})$ with the
properties that $\langle g_i,\mu^i\rangle=a_i$ and
$\sigma^i-\mu^i\geqslant0$ for all $i\in I$.

Along with its electrostatic interpretation, such a problem has
found various important applications in approximation theory (see,
e.g.,~\cite{D,DS,R}).

The main question is whether
minimizers~$\lambda^{\boldsymbol{\sigma}}_{\mathbf{A}}$ in the
constrained minimal $\mathbf{f}$-weighted energy problem exist. If
$\mathbf{A}$ is finite, all the~$A_i$ are compact, $\kappa(x,y)$ is
continuous on $A_\ell\times A_j$ whenever $\alpha_\ell\ne\alpha_j$,
and Case~I takes place, then the existence of
those~$\lambda^{\boldsymbol{\sigma}}_{\mathbf{A}}$ can easily be
established by exploiting the vague topology only, since then the
class of admissible vector measures is vaguely compact, while
$G_{\mathbf{f}}(\boldsymbol{\mu})$ is vaguely lower semicontinuous
(cf.~\cite{GR,NS,O,ST}).

However, these arguments break down if any of the above-mentioned
four assumptions is dropped, and then the problem on the existence
of minimizers becomes rather nontrivial. In particular, the class of
admissible vector measures is no longer vaguely compact if any of
the~$A_i$ is noncompact. Another difficulty is that
$G_{\mathbf{f}}(\boldsymbol{\mu})$ might not be vaguely lower
semicontinuous when Case~II holds.

To solve the problem on the existence of
minimizers~$\lambda^{\boldsymbol{\sigma}}_{\mathbf{A}}$ in the
general case, we restrict ourselves to a positive definite
kernel~$\kappa$ and develop an approach based on the following
crucial arguments.

The set $\mathcal E^+(\mathbf{A})$ is shown to be a {\it
semimetric\/} space with the semimetric
\begin{equation}\label{vseminorm}
\|\boldsymbol{\mu}_1-\boldsymbol{\mu}_2\|_{\mathcal
E^+(\mathbf{A})}:=\Bigl[\sum_{i,j\in
I}\,\alpha_i\alpha_j\kappa(\mu^i_1-\mu^i_2,\mu^j_1-\mu^j_2)\Bigr]^{1/2},
\end{equation}
and one can define an inclusion $R$ of $\mathcal E^+(\mathbf{A})$
into the pre-Hilbert space~$\mathcal E$ such that $\mathcal
E^+(\mathbf{A})$ becomes {\em isometric\/} to its $R$-image, the
latter being regarded as a semimetric subspace of~$\mathcal E$ (see
Theorem~\ref{lemma:semimetric}). Similar to the terminology
in~$\mathcal E$, we therefore call the topology of the semimetric
space $\mathcal E^+(\mathbf{A})$ {\em strong\/}.

Another crucial fact is that, for rather general $\kappa$,
$\mathbf{g}$, and~$\mathbf{a}$, the topological subspace of
$\mathcal E^+(\mathbf{A})$ consisting of all $\boldsymbol{\mu}$ such
that $\langle g_i,\mu^i\rangle\leqslant a_i$ and
$\sigma^i-\mu^i\geqslant0$ for all $i\in I$ turns out to be {\it
strongly complete\/} (see Theorem~\ref{th:strong}).

Using these arguments, we obtain sufficient conditions for the
existence of minimizers~$\lambda^{\boldsymbol{\sigma}}_{\mathbf{A}}$
and establish statements on their uniqueness and vague compactness
(see Lemma~\ref{lemma:unique} and Theorem~\ref{exist}). Examples
illustrating the sharpness of the sufficient conditions are provided
(see Sec.~\ref{sharp}). We also analyze continuity properties
of~$\lambda^{\boldsymbol{\sigma}}_{\mathbf{A}}$ relative to the
vague and strong topologies when both $\boldsymbol{\sigma}$
and~$\mathbf A$ are varied (see Theorems~\ref{cor:sigma},
\ref{cor:cont} and Corollaries~\ref{contcor1}, \ref{contcor2}).

The results obtained hold true, e.g., for the Newtonian, Green or
Riesz kernels in~$\mathbb R^n$, $n\geqslant2$, as well as for the
restriction of the logarithmic kernel in~$\mathbb R^2$ to the open
unit disk, which is important in applications.

\section{Preliminaries: topologies, consistent and perfect
kernels}\label{sec:2}

In all that follows, we suppose the kernel~$\kappa$ to be positive
definite. In addition to the {\em strong\/} topology on~$\mathcal
E$, determined by the seminorm $\|\nu\|:=\|\nu\|_\mathcal
E:=\|\nu\|_\kappa:=\sqrt{\kappa(\nu,\nu)}$, it is often useful to
consider the {\em weak\/} topology on~$\mathcal E$, defined by means
of the seminorms $\nu\mapsto|\kappa(\nu,\mu)|$, $\mu\in\mathcal E$
(see~\cite{F1}). The Cauchy--Schwarz inequality
\begin{equation*}
|\kappa(\nu,\mu)|\leqslant\|\nu\|\,\|\mu\|,\quad\text{where\ }
\nu,\mu\in\mathcal E,\end{equation*} implies immediately that the
strong topology on $\mathcal E$ is finer than the weak one.

In~\cite{F1,F2}, B.~Fuglede introduced the following two {\it
equivalent\/} properties of consistency between the induced strong,
weak, and vague topologies on~$\mathcal E^+$:\smallskip
\begin{itemize}
\item[\rm(C$_1$)] {\em Every strong Cauchy net in
$\mathcal E^+$ converges strongly to any of its vague cluster
points;}
\item[\rm(C$_2$)] {\em Every strongly bounded and vaguely convergent net in
$\mathcal E^+$ converges weakly to the vague limit.}
\end{itemize}

\begin{definition}
Following Fuglede~\cite{F1}, we call a kernel~$\kappa$ {\it
consistent} if it satisfies either of the properties~(C$_1$)
and~(C$_2$), and {\em perfect\/} if, in addition, it is strictly
positive definite.\end{definition}

\begin{remark} One has to consider {\em nets\/} or {\em filters\/}
in~$\mathfrak M^+$ instead of sequences, since the vague topology in
general does not satisfy the first axiom of countability. We follow
Moore's and Smith's theory of convergence, based on the concept of
nets (see~\cite{MS}; cf.~also~\cite[Chap.~0]{E2} and
\cite[Chap.~2]{K}). However, if $\mathrm X$ is metrizable and
countable at infinity, then $\mathfrak M^+$ satisfies the first
axiom of countability (see~\cite[Lemma~1.2.1]{F1}) and the use of
nets may be avoided.\end{remark}

\begin{theorem}[\rm Fuglede \cite{F1}]\label{th:1} A kernel $\kappa$ is perfect if and
only if $\mathcal E^+$ is strongly complete and the strong topology
on~$\mathcal E^+$ is finer than the vague one.\end{theorem}

\begin{remark} In $\mathbb R^n$, $n\geqslant 3$, the
Newtonian kernel $|x-y|^{2-n}$ is perfect~\cite{Car}. So are the
Riesz kernels $|x-y|^{\alpha-n}$, $0<\alpha<n$, in~$\mathbb R^n$,
$n\geqslant2$~\cite{D1,D2}, and the restriction of the logarithmic
kernel $-\log\,|x-y|$ in $\mathbb R^2$ to the open unit
disk~\cite{L}. Furthermore, if $D$ is an open set in~$\mathbb R^n$,
$n\geqslant 2$, and its generalized Green function~$g_D$ exists
(see, e.g.,~\cite[Th.~5.24]{HK}), then the kernel $g_D$ is perfect
as well~\cite{E1}.\end{remark}

\begin{remark} As is seen from the above definitions and Theorem~\ref{th:1}, the concept of consistent
or perfect kernels is an efficient tool in minimal energy problems
over classes of {\em nonnegative scalar\/} Radon measures with
finite energy. Indeed, the theory of capacities of {\em sets\/} has
been developed in~\cite{F1} for exactly those kernels. We shall show
below that this concept is efficient, as well, in minimal energy
problems over classes of {\em vector measures\/} of finite or
infinite dimensions. This is guaranteed by a theorem on the strong
completeness of proper subspaces of the semimetric space~$\mathcal
E^+(\mathbf{A})$, to be stated in Sec.~\ref{sec:strong}.\end{remark}

\section{Condensers. Vector measures and their energies}

\subsection{Condensers of countably many plates. Associated vector measures}
Throughout the article, let $I^+$ and $I^-$ be fixed countable,
disjoint sets of indices, where the latter is allowed to be empty,
and let $I$ denote their union. Assume that to every $i\in I$ there
corresponds a (unique) nonempty, closed set~$A_i\subset\mathrm X$.

\begin{definition} A collection $\mathbf{A}=(A_i)_{i\in I}$ is
called an $(I^+,I^-)$-{\em condenser\/} (or simply a {\it
condenser\/}) in~$\mathrm X$ if every compact subset of~$\mathrm X$
intersects with at most finitely many~$A_i$ and
\begin{equation}
A_i\cap A_j=\varnothing\quad\text{for all \ } i\in I^+, \ j\in I^-.
\label{non}
\end{equation}\end{definition}

A condenser $\mathbf{A}$ is called {\em compact\/} if so are
all~$A_i$, $i\in I$, and {\em finite\/} if $I$ is finite. The sets
$A_i$, $i\in I^+$, and $A_j$, $j\in I^-$, are called the {\it
positive\/} and, respectively,  the {\em negative plates\/}
of~$\mathbf{A}$. (Note that any two equally sign\-ed plates can
intersect each other or even coincide.) In the sequel, also the
following notation will be used:
\begin{equation*}
A^+:=\bigcup_{i\in I^+}\,A_i,\qquad A^-:=\bigcup_{i\in I^-}\,A_i.
\end{equation*} Observe that $A^+$
and~$A^-$ might both be noncompact even for a compact~$\mathbf{A}$.

Given a condenser $\mathbf{A}$, let $\mathfrak M^+(\mathbf{A})$
consist of all nonnegative vector measures
$\boldsymbol{\mu}=(\mu^i)_{i\in I}$, where $\mu^i\in\mathfrak
M^+(A_i)$ for all $i\in I$; that is, $\mathfrak
M^+(\mathbf{A}):=\prod_{i\in I}\,\mathfrak M^+(A_i)$. The product
topology on~$\mathfrak M^+(\mathbf{A})$, where every $\mathfrak
M^+(A_i)$ is equipped with the vague topology, is likewise called
{\em vague\/}. Since the space $\mathfrak M(\mathrm X)$ is
Hausdorff, so is $\mathfrak M^+(\mathbf{A})$ (cf.~\cite[Chap.~3,
Th.~5]{K}).

A set $\mathfrak F\subset\mathfrak M^+(\mathbf{A})$ is {\em vaguely
bounded\/} if, for every $\varphi\in\mathrm C_0(\mathrm X)$ and
every $i\in I$,
\[\sup_{\mu\in\mathfrak F}\,|\mu^i(\varphi)|<\infty.\]

\begin{lemma}\label{lem:vaguecomp} If $\mathfrak F\subset\mathfrak
M^+(\mathbf{A})$ is vaguely bounded, then it is vaguely relatively
compact.\end{lemma}

\begin{proof} Since by~\cite[Chap.~III, \S~2, Prop.~9]{B2} any
vaguely bounded part of~$\mathfrak M$ is vaguely relatively compact,
the lemma follows from Tychonoff's theorem on the product of compact
spaces (see, e.g.,~\cite[Chap.~5, Th.~13]{K}).\end{proof}

\subsection{Mapping $R:\mathfrak M^+(\mathbf{A})\to\mathfrak M$.
Relation of $R$-equivalency on $\mathfrak
M^+(\mathbf{A})$}\label{sect:R}

Since each compact subset of~$\mathrm X$ intersects with at most
finitely many~$A_i$, for every $\varphi\in\mathrm C_0(\mathrm X)$
only a finite number of~$\mu^i(\varphi)$ (where
$\boldsymbol{\mu}\in\mathfrak M^+(\mathbf{A})$ is given) are
nonzero. This yields that to every vector measure
$\boldsymbol{\mu}\in\mathfrak M^+(\mathbf{A})$ there corresponds a
unique scalar Radon measure~$R\boldsymbol{\mu}\in\mathfrak M$ such
that
\[R\boldsymbol{\mu}(\varphi)=\sum_{i\in I}\,\alpha_i\mu^i(\varphi)\quad\text{for
all \ }\varphi\in\mathrm C_0(\mathrm X),\]  where
\[\alpha_i:=\left\{
\begin{array}{rll} +1 & \text{if} & i\in I^+,\\ -1 & \text{if} & i\in
I^-.\\ \end{array} \right.\] Then, because of~(\ref{non}), the
positive and negative parts in the Hahn--Jordan decomposition of
$R\boldsymbol{\mu}$ can respectively be written in the form
\[ R\boldsymbol{\mu}^+=\sum_{i\in I^+}\,\mu^i,\qquad R\boldsymbol{\mu}^-=\sum_{i\in
I^-}\,\mu^i.\]

Of course, the inclusion $R$ of $\mathfrak M^+(\mathbf{A})$ into
$\mathfrak M$, thus defined, is in general non-injective, i.e., one
may choose $\boldsymbol{\mu}_1,\boldsymbol{\mu}_2\in\mathfrak
M^+(\mathbf{A})$ so that $\boldsymbol{\mu}_1\ne\boldsymbol{\mu}_2$,
though $R\boldsymbol{\mu}_1=R\boldsymbol{\mu}_2$. We call
$\boldsymbol{\mu}_1,\boldsymbol{\mu}_2\in\mathfrak M^+(\mathbf{A})$
{\em $R$-equivalent\/} if $R\boldsymbol{\mu}_1=R\boldsymbol{\mu}_2$
--- or, which is equivalent, whenever $\sum_{i\in
I}\,\mu_1^i=\sum_{i\in I}\,\mu_2^i$.

Observe that the relation of $R$-equivalency implies that of
identity (and, hence, these two relations on~$\mathfrak
M^+(\mathbf{A})$ are actually equivalent) if and only if all~$A_i$,
$i\in I$, are mutually disjoint.

\begin{lemma}\label{lem:vague'}The vague convergence of
$(\boldsymbol{\mu}_s)_{s\in S}\subset\mathfrak M^+(\mathbf{A})$
to~$\boldsymbol{\mu}_0\in\mathfrak M^+(\mathbf{A})$ implies the
vague convergence of $(R\boldsymbol{\mu}_s)_{s\in S}$
to~$R\boldsymbol{\mu}_0$.\end{lemma}

\begin{proof} This is obvious in view of the fact that the
support of any $\varphi\in\mathrm C_0(\mathrm X)$ can have points in
common with only finitely many~$A_i$.\end{proof}

\begin{remark} Lemma~\ref{lem:vague'} in general can not be
inverted. However, if all the $A_i$ are mutually disjoint, then the
vague convergence of $(R\boldsymbol{\mu}_s)_{s\in S}$
to~$R\boldsymbol{\mu}_0$ implies the vague convergence of
$(\boldsymbol{\mu}_s)_{s\in S}$ to~$\boldsymbol{\mu}_0$, which is
seen by using the Tietze--Urysohn extension theorem
\cite[Th.~0.2.13]{E2}.\end{remark}

\subsection{How the energies $\kappa(\boldsymbol{\mu},\boldsymbol{\mu})$ and
$\kappa(R\boldsymbol{\mu},R\boldsymbol{\mu})$ are related to each
other?}\label{convexx} In accordance with an electrostatic
interpretation of a condenser~$\mathbf{A}$, we assume that the law
of interaction between the charges lying on the plates~$A_i$, $i\in
I$, is determined by the matrix $(\alpha_i\alpha_j)_{i,j\in I}$.
Then the {\em mutual energy\/} of
$\boldsymbol{\mu},\boldsymbol{\mu}_1\in\mathfrak M^+(\mathbf{A})$ is
given by the expression\footnote{It will be shown below
(see~Corollary~\ref{muten}) that the mutual energy is well defined
and finite (hence, the series in~(\ref{vectormutualenergy})
converges absolutely) at least for all measures from $\mathcal
E^+(\mathbf{A})$.}
\begin{equation}\label{vectormutualenergy}\kappa(\boldsymbol{\mu},\boldsymbol{\mu}_1):=\sum_{i,j\in
I}\,\alpha_i\alpha_j\kappa(\mu^i,\mu_1^j).\end{equation} For
$\boldsymbol{\mu}=\boldsymbol{\mu}_1$ the mutual energy defines the
{\em energy\/} $\kappa(\boldsymbol{\mu},\boldsymbol{\mu})$
of~$\boldsymbol{\mu}$. Let $\mathcal E^+(\mathbf{A})$ consist of all
$\boldsymbol{\mu}\in\mathfrak M^+(\mathbf{A})$ with
$-\infty<\kappa(\boldsymbol{\mu},\boldsymbol{\mu})<\infty$.

\begin{lemma}\label{enfinite}  For $\boldsymbol{\mu}\in\mathfrak M^+(\mathbf{A})$ to
have finite energy, it is necessary and sufficient that
$\mu^i\in\mathcal E$ for all $i\in I$ and $\sum_{i\in
I}\,\|\mu^i\|^2<\infty$.\end{lemma}

\begin{proof} This follows  immediately from the above definitions due to the inequality
$2\kappa(\nu_1,\nu_2)\leqslant\|\nu_1\|^2+\|\nu_2\|^2$ for
$\nu_1,\,\nu_2\in\mathcal E$.\end{proof}

In view of the convexity of $\mathfrak M^+(\mathbf{A})$,
Lemma~\ref{enfinite} yields that also $\mathcal E^+(\mathbf{A})$
forms a {\em convex cone\/}.

In order to establish relations between the mutual energies of
vector measures and those of their (scalar) $R$-images, we need the
following two lemmas, the former being well known (see,
e.g.,~\cite{F1}). In both, $\mathrm Y$ is a locally compact
Hausdorff space.

\begin{lemma}\label{lemma:lower}If\/
$\psi\in\mathrm\Phi(\mathrm Y)$ is given, then the map
$\nu\mapsto\langle\psi,\nu\rangle$ is vaguely lower semicontinuous
on $\mathfrak M^+(\mathrm Y)$.\end{lemma}

In particular, this implies that the potential $\kappa(\cdot,\nu)$
of any $\nu\in\mathfrak M^+(\mathrm X)$ belongs
to~$\mathrm\Phi(\mathrm X)$.

\begin{lemma}\label{integral} Consider an $(L^+,L^-)$-condenser $\mathbf B=(B_\ell)_{\ell\in L}$
in $\mathrm Y$, a vector measure
$\boldsymbol{\omega}=(\omega^\ell)_{\ell\in L}\in\mathfrak
M^+(\mathbf{B})$, and a function $\psi\in\mathrm\Phi(\mathrm Y)$.
For $\langle\psi,R\boldsymbol{\omega}\rangle$ to be finite, it is
necessary and sufficient that $\sum_{\ell\in
L}\,\alpha_\ell\langle\psi,\omega^\ell\rangle$ converge absolutely,
and then
\begin{equation*}
\langle\psi,R\boldsymbol{\omega}\rangle=\sum_{\ell\in
L}\,\alpha_\ell\langle\psi,\mu^\ell\rangle.\label{lemma11}
\end{equation*}
\end{lemma}

\begin{proof} We can assume $\psi$ to be nonnegative,
for if not, we replace $\psi$ by a function~$\psi'\geqslant0$
obtained by adding to~$\psi$ a suitable constant~$c>0$, which is
always possible since a lower semicontinuous function is bounded
from below on a compact space. Hence,
\[\langle\psi,R\boldsymbol{\omega}^+\rangle\geqslant\sum_{\ell\in L^+, \ \ell\leqslant N}\,\langle
\psi,\omega^\ell\rangle\quad\text{for all \ }N\in L^+.\] On the
other hand, the sum of $\omega^\ell$ over all $\ell\in L^+$ that do
not exceed~$N$ approaches~$R\boldsymbol{\omega}^+$ vaguely as
$N\to\infty$; consequently, by Lemma~\ref{lemma:lower},
\begin{equation*}
\langle\psi,R\boldsymbol{\omega}^+\rangle\leqslant\lim_{N\to\infty}\,\sum_{\ell\in
L^+, \ \ell\leqslant
N}\,\langle\psi,\omega^\ell\rangle.\end{equation*} Combining the
last two inequalities and then letting $N\to\infty$, we get
\[\langle\psi,R\boldsymbol{\omega}^+\rangle=\sum_{\ell\in L^+}\,\langle\psi,\omega^\ell\rangle.\]
Since the same holds true for $R\boldsymbol{\omega}^-$ and~$L^-$
instead of $R\boldsymbol{\omega}^+$ and~$L^+$, the lemma
follows.\end{proof}

To apply Lemma~\ref{integral} to the condenser $\mathbf
A\times\mathbf A:=(A_i\times A_j)_{(i,j)\in I\times I}$ in $\mathrm
X\times\mathrm X$ with $\alpha_{(i,j)}:=\alpha_i\alpha_j$, we
observe that any $\boldsymbol{\omega}\in\mathfrak M^+(\mathbf
A\times\mathbf A)$ can be written as
$\boldsymbol{\mu}\otimes\boldsymbol{\mu}_1:=(\mu^i\otimes\mu_1^j)_{(i,j)\in
I\times I}$, where $\boldsymbol{\mu},\boldsymbol{\mu}_1\in\mathfrak
M^+(\mathbf A)$. Therefore,
\[R(\boldsymbol{\mu}\otimes\boldsymbol{\mu}_1)=\sum_{i,j\in
I}\,\alpha_i\alpha_j\mu^i\otimes\mu_1^j=R\boldsymbol{\mu}\otimes
R\boldsymbol{\mu}_1.\] If, moreover, $\psi=\kappa\in\Phi(\mathrm
X\times\mathrm X)$, then we arrive at the following assertion.

\begin{corollary}\label{unknown}Given $\boldsymbol{\mu},\boldsymbol{\mu}_1\in\mathfrak M^+(\mathbf
A)$, we $\kappa(\boldsymbol{\mu},\boldsymbol{\mu}_1)=
\kappa(R\boldsymbol{\mu},R\boldsymbol{\mu}_1)$, the identity being
understood in the sense that each of its sides is finite whenever so
is the other and then they coincide.
\end{corollary}

Hence, $\boldsymbol{\mu}\in\mathfrak M^+(\mathbf{A})$ belongs to
$\mathcal E^+(\mathbf A)$ if and only if
$R\boldsymbol{\mu}\in\mathcal E$ and, furthermore,
\begin{equation}\label{Ren}
\kappa(\boldsymbol{\mu},\boldsymbol{\mu})=\kappa(R\boldsymbol{\mu},R\boldsymbol{\mu})\quad\text{for
all \ }\boldsymbol{\mu}\in\mathcal E^+(\mathbf A).
\end{equation}
In view of the positive definiteness of the kernel, this yields the
following property of positivity of the energy
$\kappa(\boldsymbol{\mu},\boldsymbol{\mu})$, which was not obvious a
priori.

\begin{corollary}\label{posen}For all $\boldsymbol{\mu}\in\mathcal E^+(\mathbf
A)$, it is true that
$\kappa(\boldsymbol{\mu},\boldsymbol{\mu})\geqslant0$.\end{corollary}

\begin{corollary}\label{muten}For any $\boldsymbol{\mu},\boldsymbol{\mu}_1\in\mathcal E^+(\mathbf
A)$, we have
\begin{equation}\label{vecen}
\kappa(\boldsymbol{\mu},\boldsymbol{\mu}_1)=\kappa(R\boldsymbol{\mu},R\boldsymbol{\mu}_1)=\sum_{i,j\in
I}\,\alpha_i\alpha_j\kappa(\mu^i,\mu_1^j),
\end{equation}
and the series here converges absolutely.
\end{corollary}

\begin{proof}For any $\boldsymbol{\mu},\boldsymbol{\mu}_1\in\mathcal E^+(\mathbf
A)$, we get $R\boldsymbol{\mu},R\boldsymbol{\mu}_1\in\mathcal E$;
hence, $\kappa(R\boldsymbol{\mu},R\boldsymbol{\mu}_1)$ is finite.
Therefore, repeated application of Corollary~\ref{unknown} gives the
desired conclusion.\end{proof}

\subsection{Semimetric space of vector measures of finite
energy}\label{sec:semimetric}

\begin{theorem}\label{lemma:semimetric} $\mathcal E^+(\mathbf{A})$ forms a semimetric space with
the semimetric \mbox{$\|\cdot\|_{\mathcal E^+(\mathbf{A})}$},
defined by~{\rm(\ref{vseminorm})}, and this space is isometric to
its $R$-image. The semimetric $\|\cdot\|_{\mathcal E^+(\mathbf{A})}$
is a metric if and only if the kernel $\kappa$ is strictly positive
definite while all $A_i$, $i\in I$, are mutually disjoint.
\end{theorem}

\begin{proof}Fix $\boldsymbol{\mu}_1,\boldsymbol{\mu}_2\in\mathcal E^+(\mathbf{A})$.
Applying Corollary~\ref{muten} to
$\kappa(R\boldsymbol{\mu}_k,R\boldsymbol{\mu}_t)$, $k,t=1,2$, we get
\begin{equation*}\label{isometric}\|R\boldsymbol{\mu}_1-R\boldsymbol{\mu}_2\|_{\mathcal
E}^2=\sum_{i,j\in
I}\,\alpha_i\alpha_j\kappa(\mu^i_1-\mu^i_2,\mu^j_1-\mu^j_2),\end{equation*}
where the series converges absolutely. Compared
with~(\ref{vseminorm}), this relation yields
\begin{equation}\label{seminorm}\|\boldsymbol{\mu}_1-\boldsymbol{\mu}_2\|_{\mathcal E^+(\mathbf{A})}=
\|R\boldsymbol{\mu}_1-R\boldsymbol{\mu}_2\|_{\mathcal
E}.\end{equation} Since $\|\cdot\|_{\mathcal E}$ is a seminorm
on~$\mathcal E$, the theorem follows.\end{proof}

From now on, $\mathcal E^+(\mathbf{A})$ will always be treated as a
semimetric space with the semimetric $\|\cdot\|:=\|\cdot\|_{\mathcal
E^+(\mathbf{A})}$. Since $\mathcal E^+(\mathbf{A})$ and its
$R$-image are isometric, similar to the terminology in~$\mathcal E$
we shall call the topology on~$\mathcal E^+(\mathbf{A})$ {\it
strong\/}.

Two elements of~$\mathcal E^+(\mathbf{A})$, $\boldsymbol{\mu}_1$
and~$\boldsymbol{\mu}_2$, are said to be {\em equivalent in\/}
$\mathcal E^+(\mathbf{A})$ if
$\|\boldsymbol{\mu}_1-\boldsymbol{\mu}_2\|=0$. Observe that the
equivalence in~$\mathcal E^+(\mathbf{A})$ implies $R$-equivalence
(i.e., then $R\boldsymbol{\mu}_1=R\boldsymbol{\mu}_2$) provided the
kernel~$\kappa$ is strictly positive definite, and it implies the
identity (i.e., then $\boldsymbol{\mu}_1=\boldsymbol{\mu}_2$) if,
moreover, all~$A_i$, $i\in I$, are mutually disjoint.

\section{Constrained minimal\/ $\mathbf{f}$-weighted
energy problem}\label{sec:statement}
\subsection{Statement of the problem}

Consider an external field $\mathbf{f}=(f_i)_{i\in I}$ satisfying
Case~I or Case~II (see the Introduction), and assume each~$f_i$ to
affect the charges on~$A_i$ only. The $\mathbf{f}$-{\em weight\-ed
energy\/} $G_{\mathbf{f}}(\boldsymbol{\mu})$ of
$\boldsymbol{\mu}\in\mathfrak M^+(\mathbf{A})$ is defined
by~(\ref{wen}), and let $\mathcal E_{\mathbf f}^+(\mathbf{A})$
consist of all $\boldsymbol{\mu}\in\mathcal E^+(\mathbf{A})$ with
$-\infty<G_{\mathbf{f}}(\boldsymbol{\mu})<\infty$.

Also fix a nonnegative vector measure
$\boldsymbol{\sigma}\in\mathfrak M^+(\mathbf{A})$, called a {\it
constraint associated with\/}~$\mathbf A$, a numerical vector
$\mathbf{a}=(a_i)_{i\in I}$ with $a_i>0$, and a vector-valued
function $\mathbf{g}=(g_i)_{i\in I}$, where all the $g_i:\mathrm
X\to(0,\infty)$ are continuous. We define
\begin{equation*}\mathfrak M^+_{\boldsymbol{\sigma}}(\mathbf{A}):=\bigl\{\boldsymbol{\mu}\in\mathfrak
M^+(\mathbf{A}): \
\boldsymbol{\mu}\leqslant\boldsymbol{\sigma}\bigr\},\end{equation*}
where $\boldsymbol{\mu}\leqslant\boldsymbol{\sigma}$ means that
$\sigma^i-\mu^i\geqslant0$ for all $i\in I$, and
\begin{equation*}\mathfrak M^+_{\boldsymbol{\sigma}}(\mathbf{A},\mathbf{a},\mathbf{g}):=
\bigl\{\boldsymbol{\mu}\in\mathfrak
M^+_{\boldsymbol{\sigma}}(\mathbf{A}): \ \langle
g_i,\mu^i\rangle=a_i\quad\text{for all \ } i\in
I\bigr\},\end{equation*}
\[\mathcal E^+_{\boldsymbol{\sigma}}(\mathbf{A},\mathbf{a},\mathbf{g}):=
\mathfrak
M^+_{\boldsymbol{\sigma}}(\mathbf{A},\mathbf{a},\mathbf{g})\cap\mathcal
E^+(\mathbf{A}),\]
\[\mathcal E^+_{\boldsymbol{\sigma},\mathbf f}(\mathbf{A},\mathbf{a},\mathbf{g}):=\mathfrak
M^+_{\boldsymbol{\sigma}}(\mathbf{A},\mathbf{a},\mathbf{g})\cap\mathcal
E_{\mathbf f}^+(\mathbf{A})\] and then we introduce the extremal
value
\begin{equation}\label{G}G^{\boldsymbol{\sigma}}_{\mathbf{f}}(\mathbf{A},\mathbf{a},\mathbf{g}):=
\inf_{\mu\in\mathcal E^+_{\boldsymbol{\sigma},\mathbf
f}(\mathbf{A},\mathbf{a},\mathbf{g})}\,G_{\mathbf{f}}(\boldsymbol{\mu}).
\end{equation}
In (\ref{G}), as usual, the infimum over the empty set is taken to
be~$+\infty$.

If $\mathcal E^+_{\boldsymbol{\sigma},\mathbf
f}(\mathbf{A},\mathbf{a},\mathbf{g})$ is nonempty or, which is
equivalent, if it is true that\footnote{See
Lemma~\ref{lemma:finite.necsuf} below for necessary and (or)
sufficient conditions for (\ref{nonzero1}) to hold. Then, actually,
$G^{\boldsymbol{\sigma}}_{\mathbf{f}}(\mathbf{A},\mathbf{a},\mathbf{g})$
has to be finite (see Corollary~\ref{lemma:minusfinite}).}
\begin{equation}\label{nonzero1}
G^{\boldsymbol{\sigma}}_{\mathbf{f}}(\mathbf{A},\mathbf{a},\mathbf{g})<\infty,\end{equation}
then the following problem makes sense.

\begin{problem}
Does there exist
$\boldsymbol{\lambda}^{\boldsymbol{\sigma}}_{\mathbf{A}}\in\mathcal
E^+_{\boldsymbol{\sigma},\mathbf
f}(\mathbf{A},\mathbf{a},\mathbf{g})$ with
$G_{\mathbf{f}}(\boldsymbol{\lambda}^{\boldsymbol{\sigma}}_{\mathbf{A}})=
G^{\boldsymbol{\sigma}}_{\mathbf{f}}(\mathbf{A},\mathbf{a},\mathbf{g})$?\end{problem}

Along with its electrostatic interpretation, such a problem has
found various important applications in approximation theory (see,
e.g., \cite{D,DS,R}). The problem is called {\em solvable\/} if the
class $\mathfrak
S^{\boldsymbol{\sigma}}_{\mathbf{f}}(\mathbf{A},\mathbf{a},\mathbf{g})$
of all the minimizers
$\boldsymbol{\lambda}=\boldsymbol{\lambda}^{\boldsymbol{\sigma}}_{\mathbf{A}}$
is nonempty.

\subsection{On the uniqueness of minimizers}

\begin{lemma}\label{lemma:unique}
If $\boldsymbol{\lambda}$ and $\widehat{\boldsymbol{\lambda}}$
belong to $\mathfrak
S^{\boldsymbol{\sigma}}_{\mathbf{f}}(\mathbf{A},\mathbf{a},\mathbf{g})$,
then
\begin{equation}\label{uniq1}
\|\boldsymbol{\lambda}-\widehat{\boldsymbol{\lambda}}\|_{\mathcal
E^+(\mathbf{A})}=0.\end{equation}
\end{lemma}

\begin{proof} It follows from the convexity of $\mathcal E^+(\mathbf A)$ (see~Sec.~\ref{convexx})
that so is $\mathcal E^+_{\boldsymbol{\sigma},\mathbf
f}(\mathbf{A},\mathbf{a},\mathbf{g})$, which makes it possible to
conclude from (\ref{wen}), (\ref{Ren}), and (\ref{G}) that
\[4G^{\boldsymbol{\sigma}}_{\mathbf{f}}(\mathbf{A},\mathbf{a},\mathbf{g})\leqslant
4G_{\mathbf{f}}\Bigl(\frac{\boldsymbol{\lambda}+\widehat{\boldsymbol{\lambda}}}{2}\Bigr)=
\|R\boldsymbol{\lambda}+R\widehat{\boldsymbol{\lambda}}\|^2+
4\langle\mathbf{f},\boldsymbol{\lambda}+\widehat{\boldsymbol{\lambda}}\rangle.\]
On the other hand, applying the parallelogram identity in the
pre-Hilbert space~$\mathcal E$ to $R\boldsymbol{\lambda}$
and~$R\widehat{\boldsymbol{\lambda}}$ and then adding and
subtracting
$4\langle\mathbf{f},\boldsymbol{\lambda}+\widehat{\boldsymbol{\lambda}}\rangle$,
we get
\[\|R\boldsymbol{\lambda}-R\widehat{\boldsymbol{\lambda}}\|^2=
-\|R\boldsymbol{\lambda}+R\widehat{\boldsymbol{\lambda}}\|^2-4\langle\mathbf{f},\boldsymbol{\lambda}+
\widehat{\boldsymbol{\lambda}}\rangle+2G_{\mathbf{f}}(\boldsymbol{\lambda})+
2G_{\mathbf{f}}(\widehat{\boldsymbol{\lambda}}).\] When combined
with the preceding relation, this yields
\[0\leqslant\|R\boldsymbol{\lambda}-R\widehat{\boldsymbol{\lambda}}\|^2\leqslant-
4G^{\boldsymbol{\sigma}}_{\mathbf{f}}(\mathbf{A},\mathbf{a},\mathbf{g})+2G_{\mathbf{f}}(\boldsymbol{\lambda})+
2G_{\mathbf{f}}(\widehat{\boldsymbol{\lambda}})=0,\] which
establishes (\ref{uniq1}) because of~(\ref{seminorm}).\end{proof}

Thus, any two minimizers (if exist) are equivalent in $\mathcal
E^+(\mathbf{A})$. Consequently, they are $R$-equiv\-alent if the
kernel~$\kappa$ is strictly positive definite, and they are equal
if, moreover, all~$A_i$, $i\in I$, are mutually disjoint.

\section{Elementary properties of
$G^{\boldsymbol{\sigma}}_{\mathbf{f}}(\mathbf{A},\mathbf{a},\mathbf{g})$}
Before analyzing the existence of minimizers and their continuity,
we provide some auxiliary results, to be needed in the sequel. Write
\[g_{i,\inf}:=\inf_{x\in A_i}\,g_i(x),\qquad g_{i,\sup}:=\sup_{x\in A_i}\,g_i(x).\]

\subsection{Monotonicity of
$G^{\boldsymbol{\sigma}}_{\mathbf{f}}(\mathbf{A},\mathbf{a},\mathbf{g})$}

On the collection of all $(I^+,I^-)$-condensers in~$\mathrm X$, it
is natural to introduce an ordering relation~$\leqslant$ by
declaring $\mathbf{A}'\leqslant\mathbf{A}$ to mean that $A_i'\subset
A_i$ for all $i\in I$. Here, $\mathbf{A}'=(A_i')_{i\in I}$. If now
$\boldsymbol{\sigma}$ is a constraint associated with~$\mathbf{A}$
and $\boldsymbol{\sigma}'$ is that associated with~$\mathbf{A}'$,
then we write
$(\mathbf{A}',\boldsymbol{\sigma}')\leqslant(\mathbf{A},\boldsymbol{\sigma})$
provided $\mathbf{A}'\leqslant\mathbf{A}$ and
$\boldsymbol{\sigma}'\leqslant\boldsymbol{\sigma}$. Then
$G^{\boldsymbol{\sigma}}_{\mathbf{f}}(\mathbf{A},\mathbf{a},\mathbf{g})$
is a nonincreasing function of $(\mathbf{A},\boldsymbol{\sigma})$,
namely
\begin{equation}
G^{\boldsymbol{\sigma}}_{\mathbf{f}}(\mathbf{A},\mathbf{a},\mathbf{g})\leqslant
G^{\boldsymbol{\sigma}'}_{\mathbf{f}}(\mathbf{A'},\mathbf{a},\mathbf{g})\quad\text{whenever
\
}(\mathbf{A}',\boldsymbol{\sigma}')\leqslant(\mathbf{A},\boldsymbol{\sigma}).
\label{increas'}
\end{equation}

We shall employ the technique of exhaustion of~$\mathbf{A}$ by
compact~$\mathbf{K}$. In doing so, we shall need the following
notation and elementary lemma.

Given $\mathbf{A}$, let $\{\mathbf{K}\}_{\mathbf{A}}$ stand for the
increasing family of all compact condensers $\mathbf{K}=(K_i)_{i\in
I}$ such that $\mathbf{K}\leqslant\mathbf{A}$. For any
$\boldsymbol{\mu}\in\mathfrak M^+(\mathbf{A})$ and
$\mathbf{K}\in\{\mathbf{K}\}_{\mathbf{A}}$, let $\mu^i_{\mathbf{K}}$
denote the trace of~$\mu^i$ upon~$K_i$, i.e.
$\mu^i_{\mathbf{K}}:=\mu_{K_i}^i$, and let
$\boldsymbol{\mu}_{\mathbf{K}}:=(\mu_{\mathbf{K}}^i)_{i\in I}$.
Observe that, if $\boldsymbol{\sigma}$ is a constraint associated
with~$\mathbf{A}$, then
$\boldsymbol{\sigma}_{\mathbf{K}}=(\sigma^i_{\mathbf{K}})_{i\in I}$
is that associated with~$\mathbf{K}$. We further write
$\widehat{\boldsymbol{\mu}}_{\mathbf{K}}:=(\hat{\mu}_{\mathbf{K}}^i)_{i\in
I}$, where
\begin{equation}\label{hatlambda''}
\hat{\mu}^i_{\mathbf{K}}:=\frac{a_i}{\langle
g_i,\mu_{\mathbf{K}}^i\rangle}\,\mu_{\mathbf{K}}^i.\end{equation}

\begin{lemma}\label{exhaustion}Fix $\boldsymbol{\mu}\in\mathcal E^+_{\boldsymbol{\sigma},\mathbf
f}(\mathbf{A},\mathbf{a},\mathbf{g})$. For every $\varepsilon>0$,
there exists\/ $\mathbf{K}_0\in\{\mathbf{K}\}_{\mathbf{A}}$ such
that, for all\/ $\mathbf{K}\in\{\mathbf{K}\}_{\mathbf{A}}$ that
follow\/ $\mathbf{K}_0$,
\begin{equation}\label{exh1}\widehat{\boldsymbol{\mu}}_{\mathbf{K}}\in\mathcal
E^+_{(1+\varepsilon)\boldsymbol{\sigma}_{\mathbf{K}},\mathbf
f}(\mathbf{K},\mathbf{a},\mathbf{g}).\end{equation}
\end{lemma}

\begin{proof}Application of \cite[Lemma~1.2.2]{F1} yields
\begin{align}
\langle g_i,\mu^i\rangle&=\lim_{\mathbf{K}\uparrow\mathbf{A}}\,\langle g_i,\mu_{\mathbf{K}}^i\rangle,
\qquad i\in I,\label{wnew}\\
\langle f_i,\mu^i\rangle&=\lim_{\mathbf{K}\uparrow\mathbf{A}}\,\langle f_i,\mu_{\mathbf{K}}^i\rangle,\qquad
i\in I,\label{w'new}\\
\kappa(\mu^i,\mu^j)&=\lim_{\mathbf{K}\uparrow\mathbf{A}}\,\kappa(\mu_{\mathbf{K}}^i,\mu_{\mathbf{K}}^j),\quad
i,j\in I.\label{wwnew}
\end{align}
Fix $\varepsilon>0$. By (\ref{wnew})--(\ref{wwnew}), for every $i\in
I$ one can choose a compact set $K_i^0\subset A_i$ so that, for all
compact sets~$K_i$ with the property $K_i^0\subset K_i\subset A_i$,
\begin{equation}
\frac{a_i}{\langle
g_i,\mu^i_{K_i}\rangle}<1+\varepsilon\,i^{-2},\label{unif2new}
\end{equation}
\begin{equation}
\bigl|\langle f_i,\mu^i\rangle-\langle
f_i,\mu^i_{K_i}\rangle\bigr|<\varepsilon\,i^{-2},\label{unif2'new}
\end{equation}
\begin{equation}
\bigl|\|\mu^i\|^2-\|\mu^i_{K_i}\|^2\bigr|<\varepsilon^2i^{-4}.\label{unif1new}
\end{equation}
Having denoted $\mathbf{K}_0:=(K_i^0)_{i\in I}$, for every
$\mathbf{K}\in\{\mathbf{K}\}_{\mathbf{A}}$ that
follows~$\mathbf{K}_0$ we get
\[\widehat{\boldsymbol{\mu}}_{\mathbf{K}}\in\mathcal
E^+_{(1+\varepsilon)\boldsymbol{\sigma}_{\mathbf{K}}}(\mathbf{K},\mathbf{a},\mathbf{g}),\]
the finiteness of the energy being obtained
from~(\ref{hatlambda''}), (\ref{unif2new}), and (\ref{unif1new})
with help of Lemma~\ref{enfinite}. Furthermore, since $\sum_{i\in
I}\,\langle f_i,\mu^i\rangle$ converges absolutely, we conclude from
(\ref{unif2new}) and~(\ref{unif2'new}) that so does $\sum_{i\in
I}\,\langle f_i,\hat{\mu}^i_{\mathbf{K}}\rangle$. This
means~(\ref{exh1}) as claimed.\end{proof}

\subsection{It is true that
$G^{\boldsymbol{\sigma}}_{\mathbf{f}}(\mathbf{A},\mathbf{a},\mathbf{g})>-\infty$}\label{secInf}

To prove the estimate, announced in the title, we need the following
two lemmas.

\begin{lemma}\label{caseii} Let Case\/ {\rm II} take place, i.e., let $f_i=\alpha_i\kappa(\cdot,\zeta)$ for all
$i\in I$, where $\zeta\in\mathcal E$ is given. Then the classes
$\mathcal E^+(\mathbf A)$ and $\mathcal E_{\mathbf f}^+(\mathbf{A})$
coincide and, furthermore,
\begin{equation}\label{yusss}
G_{\mathbf{f}}(\boldsymbol{\mu})=
\|R\boldsymbol{\mu}+\zeta\|^2-\|\zeta\|^2\quad\text{for all \ }
\boldsymbol{\mu}\in\mathcal E^+(\mathbf A).\end{equation}\end{lemma}

\begin{proof}Applying Lemma~\ref{integral} to $\boldsymbol{\mu}\in\mathcal
E^+(\mathbf{A})$ and each of $\kappa(\cdot,\zeta^+)$ and
$\kappa(\cdot,\zeta^-)$, we get
\begin{equation}\label{yu}
\langle\mathbf{f},\boldsymbol{\mu}\rangle=\sum_{i\in
I}\,\alpha_i\int\kappa(x,\zeta)\,d\mu^i(x)=\kappa(\zeta,R\boldsymbol{\mu}),
\end{equation}
where the series converges absolutely. Hence,
$\boldsymbol{\mu}\in\mathcal E_{\mathbf f}^+(\mathbf{A})$. Now,
substituting (\ref{Ren}) and~(\ref{yu}) into~(\ref{wen})
gives~(\ref{yusss}) as required.\end{proof}

\begin{lemma}\label{casei}Consider a condenser $\mathbf{B}=(B_\ell)_{\ell\in L}$
in a locally compact space~$\mathrm Y$,
$\mathbf{u}=(u_\ell)_{\ell\in L}$ with $u_\ell\in\mathrm\Phi(\mathrm
Y)$, and $\mathfrak F\subset\mathfrak M^+(\mathbf{B})$ with the
property that
\begin{equation}\label{L}\sup_{\boldsymbol{\omega}\in\mathfrak F}\,\omega^\ell(\mathrm
Y)<\infty\quad\text{for all \ }\ell\in L\end{equation} unless
$\mathrm Y$ is noncompact. Then
$\langle\mathbf{u},\boldsymbol{\omega}\rangle$ is well defined for
all\/ $\boldsymbol{\omega}\in\mathfrak F$, and
\[-\infty<\inf_{\boldsymbol{\omega}\in\mathfrak F}\,\langle\mathbf{u},\boldsymbol{\omega}\rangle
\leqslant\infty.\]
\end{lemma}

\begin{proof} We can assume $\mathrm Y$ to be compact, for if not, then
$u_\ell\geqslant0$ for all \mbox{$\ell\in L$} and the lemma is
obvious. But then $\mathbf{B}$ is to be finite while every $u_\ell$,
being lower semicontinuous, is bounded from below by~$-c_\ell$,
where $0<c_\ell<\infty$. Hence, by~(\ref{L}),
\[-\infty<-c_\ell\sup_{\boldsymbol{\omega}\in\mathfrak F}\,\omega^\ell(\mathrm
Y)\leqslant\langle u_\ell,\omega^\ell\rangle\leqslant\infty,\] which
in view of the finiteness of~$L$ yields the lemma.\end{proof}

\begin{corollary}\label{lemma:minusfinite}
$G^{\boldsymbol{\sigma}}_{\mathbf{f}}(\mathbf{A},\mathbf{a},\mathbf{g})>-\infty$.
\end{corollary}

\begin{proof} We can consider Case~I, since otherwise the corollary follows
from~(\ref{yusss}). Then $f_i\in\mathrm\Phi(\mathrm X)$ for all
$i\in I$. Furthermore, if $\mathrm X$ is compact, then
$g_{i,\inf}>0$ and
\[\sup_{\boldsymbol{\mu}\in\mathfrak
M^+_{\boldsymbol{\sigma}}(\mathbf{A},\mathbf{a},\mathbf{g})}\,\mu^i(\mathrm
X)\leqslant a_ig_{i,\inf}^{-1}<\infty.\] By Lemma~\ref{casei},
\begin{equation}\label{boundbelow}
-\infty<M_0\leqslant\langle\mathbf{f},\boldsymbol{\mu}\rangle\leqslant\infty\quad\text{for
all \ } \boldsymbol{\mu}\in\mathfrak
M^+_{\boldsymbol{\sigma}}(\mathbf{A},\mathbf{a},\mathbf{g}),\end{equation}
which together with Corollary~\ref{posen} completes the
proof.\end{proof}

\subsection{When does
$G^{\boldsymbol{\sigma}}_{\mathbf{f}}(\mathbf{A},\mathbf{a},\mathbf{g})<\infty$
hold?}\label{nonz}

Let $C(E)=C_\kappa(E)$ denote the interior capacity of a set
$E\subset\mathrm X$ relative to the kernel~$\kappa$ (see~\cite{F1}).

The following assertion provides necessary and (or) sufficient
conditions for relation~(\ref{nonzero1}) to hold (or, which is
equivalent, for $\mathcal E^+_{\boldsymbol{\sigma},\mathbf
f}(\mathbf{A},\mathbf{a},\mathbf{g})$ to be nonempty).

\begin{lemma}\label{lemma:finite.necsuf} If {\rm(\ref{nonzero1})} is
true, then necessarily
\begin{equation*}\label{nec}
C\bigl(\{x\in A_i: \ |f_i(x)|<\infty\}\bigr)>0\quad\text{for all \
}i\in I.\end{equation*} In the case where
\begin{equation}\label{s}\sum_{i\in
I}\,a_i\,g_{i,\inf}^{-1}<\infty,\end{equation} for
{\rm(\ref{nonzero1})} to hold, it is sufficient that the following
conditions be both satisfied:\smallskip
\begin{itemize}
\item[\rm(a)] for every $\mathbf{K}\in\{\mathbf{K}\}_{\mathbf{A}}$, $\boldsymbol{\sigma}_{\mathbf{K}}$ has finite
energy;\smallskip
\item[\rm(b)] there exists $M\in(0,\infty)$ not depending on~$i$ and such that
$\langle g_i,\sigma^i_{A_i^M}\rangle>a_i$, where $A^M_i:=\{x\in A_i:
\ |f_i(x)|\leqslant M\}$, $i\in I$.
\end{itemize}
\end{lemma}

\begin{proof}
To prove the necessity part of the lemma, fix
$\boldsymbol{\mu}\in\mathcal E^+_{\boldsymbol{\sigma},\mathbf
f}(\mathbf{A},\mathbf{a},\mathbf{g})$; then, by
Lemma~\ref{exhaustion}, $\widehat{\boldsymbol{\mu}}_{\mathbf{K}}$
has finite $\mathbf{f}$-weighted energy provided
$\mathbf{K}\in\{\mathbf{K}\}_{\mathbf{A}}$ is sufficiently large.
Suppose, contrary to our claim, that $C\bigl(\{x\in A_{i_0}:\
|f_{i_0}(x)|<\infty\}\bigr)=0$ for some $i_0\in I$. Since
$\hat{\mu}^{i_0}_{\mathbf{K}}$ has finite energy and is compactly
supported in~$A_{i_0}$, \cite[Lemma~2.3.1]{F1} shows that
$|f_{i_0}(x)|=\infty$ holds $\hat{\mu}^{i_0}_{\mathbf{K}}$-al\-most
everywhere ($\hat{\mu}^{i_0}_{\mathbf{K}}$-a.e.) in~$\mathrm X$.
This is impossible, for $\hat{\mu}^{i_0}_{\mathbf{K}}$ is nonzero
while
$\langle\mathbf{f},\widehat{\boldsymbol{\mu}}_{\mathbf{K}}\rangle$
is finite.

To establish the sufficient part, suppose (\ref{s}), (a), and~(b) to
be satisfied. Then for every $i\in I$ one can choose a compact set
$K_i\subset A^M_i$ so that
\begin{equation}\label{finitesuff1}\langle
g_i,\sigma^i_{K_i}\rangle>a_i,\end{equation} which is seen from~(b)
due to~\cite[Lemma~1.2.2]{F1}. Having denoted
$\mathbf{K}:=(K_i)_{i\in I}$, we consider the vector measure
$\widehat{\boldsymbol{\sigma}}_{\mathbf{K}}$ with the components
$\hat{\sigma}^i_{\mathbf{K}}$, defined by~(\ref{hatlambda''}) with
$\sigma_{\mathbf{K}}^i$ in place of $\mu_{\mathbf{K}}^i$. It follows
from~(\ref{finitesuff1}) that
$\widehat{\boldsymbol{\sigma}}_{\mathbf{K}}\in\mathcal
E^+_{{\boldsymbol{\sigma}_{\mathbf{K}}}}(\mathbf{K},\mathbf{a},\mathbf{g})$,
the finiteness of the energy being obtained from~(a) in view of
Lemma~\ref{enfinite}. Furthermore, since
\[\sum_{i\in I}\,\langle|f_i|,\hat{\sigma}^i_{\mathbf{K}}\rangle\leqslant M\,\sum_{i\in I}\,
\frac{a_i\sigma^i_{\mathbf{K}}(\mathrm X)}{\langle
g_i,\sigma^i_{\mathbf{K}}\rangle}\leqslant M\,\sum_{i\in
I}\,a_i\,g_{i,\inf}^{-1},\]  we actually have
$\widehat{\boldsymbol{\sigma}}_{\mathbf{K}}\in\mathcal
E^+_{\boldsymbol{\sigma}_{\mathbf{K}},\mathbf{f}}(\mathbf{K},\mathbf{a},\mathbf{g})$
by~(\ref{s}), and so
$G_{\mathbf{f}}^{\sigma_{\mathbf{K}}}(\mathbf{K},\mathbf{a},\mathbf{g})<\infty$.
Since
$(\mathbf{K},\boldsymbol{\sigma}_{\mathbf{K}})\leqslant(\mathbf{A},\boldsymbol{\sigma})$,
this together with~(\ref{increas'}) yields~(\ref{nonzero1}) as was
to be proved.\end{proof}

\begin{remark}If $\mathbf{A}$ is finite, then Lemma~\ref{lemma:finite.necsuf} remains
true with (b) replaced by the following assumption: for every $i\in
I$, $\langle g_i,\sigma^i\rangle>a_i$ while $|f_i|\ne\infty$ locally
$\sigma^i$-a.e. (see \cite[Lemma 4]{Z6}).\end{remark}

\section{Main results}\label{main}
From now on, (\ref{nonzero1}) is always assumed to hold. Observe
that, according to Corollary~\ref{lemma:minusfinite},
$G^{\boldsymbol{\sigma}}_{\mathbf{f}}(\mathbf{A},\mathbf{a},\mathbf{g})$
is then actually finite.

Suppose for a moment that the condenser $\mathbf{A}$ is compact.
Then the class $\mathfrak
M^+_{\boldsymbol{\sigma}}(\mathbf{A},\mathbf{a},\mathbf{g})$ is
vaguely bounded and closed and hence, by Lemma~\ref{lem:vaguecomp},
it is vaguely compact. If, moreover, $\mathbf{A}$~is finite,
$\kappa$ is continuous on~$A^+\times A^-$, and Case~I holds, then
$G_{\mathbf{f}}(\boldsymbol{\mu})$ is vaguely lower semicontinuous
on~$\mathcal E_{\mathbf{f}}^+(\mathbf{A})$ and, therefore, the
existence of minimizers~$\lambda_\mathbf{A}^{\boldsymbol{\sigma}}$
immediately follows (cf.~\cite{GR,NS,O,ST}).

However, these arguments break down if any of the above-mentioned
four assumptions is dropped, and then the problem on the existence
of minimizers~$\lambda_\mathbf{A}^{\boldsymbol{\sigma}}$ becomes
rather nontrivial. In particular, $\mathfrak
M^+_{\boldsymbol{\sigma}}(\mathbf{A},\mathbf{a},\mathbf{g})$ is no
longer vaguely compact if any of the $A_i$ is noncompact. Another
difficulty is that $G_{\mathbf{f}}(\boldsymbol{\mu})$ might not be
vaguely lower semicontinuous on~$\mathcal
E_{\mathbf{f}}^+(\mathbf{A})$ when Case~II takes place.

To solve the problem on the existence of
minimizers~$\lambda_\mathbf{A}^{\boldsymbol{\sigma}}$ in the general
case, we develop an approach based on both the vague and strong
topologies in the semimetric space~$\mathcal E^+(\mathbf{A})$,
introduced for vector measures of finite dimensions
in~\cite{Z6,Z7,Z8}. For $I=\{1\}$, see also~\cite{Z8a} (compare
with~\cite{D,DS,R}).

\subsection{Standing assumptions}\label{sec:standing}

In addition to (\ref{nonzero1}), in all that follows it is always
required that the kernel~$\kappa$ is consistent and either
$I^-=\varnothing$, or there hold (\ref{s}) and the following
condition:
\begin{equation}\sup_{x\in A^+,\ y\in
A^-}\,\kappa(x,y)<\infty.\label{bou}
\end{equation}

\begin{remark} Note that these assumptions on a kernel are not too
restrictive. In particular, they all are satisfied by the Newtonian,
Riesz, or Green kernels in~$\mathbb R^n$, $n\geqslant2$, provided
the Euclidean distance between $A^+$ and $A^-$ is nonzero, as well
as by the restriction of the logarithmic kernel in~$\mathbb R^2$ to
the open unit disk.\end{remark}

\subsection{Minimizers: existence and vague compactness}\label{sec:ex}
A proposition $u(x)$ involving a variable point $x\in\mathrm X$ is
said to subsist {\em nearly everywhere\/}~(n.e.) in~$E$, where
$E\subset\mathrm X$, if the set of all $x\in E$ for which $u$ fails
to hold is of interior capacity zero.

\begin{theorem}\label{exist} Under the standing assumptions, suppose, moreover, for every $i\in
I$ the following {\rm(a)--(c)} to hold:\smallskip
\begin{itemize}
\item[\rm(a)] Either $g_{i,\inf}>0$ or
$A_i$ can be written as a countable union of compact sets;\smallskip
\item[\rm(b)] Either $g_{i,\sup}<\infty$ or there exist\/ $r_i\in(1,\infty)$
and\/ $\tau_i\in\mathcal E$ with the property
\begin{equation}
g_i^{r_i}(x)\leqslant\kappa(x,\tau_i)\quad\text{n.e. in \ } A_i;
\label{growth}
\end{equation}
\item[\rm(c)] $A_i$ either is compact or has finite interior
capacity.
\end{itemize}\smallskip
Then, for any $\boldsymbol{\sigma}$, $\mathbf{f}$, and\/
$\mathbf{a}$, $\mathfrak
S^{\boldsymbol{\sigma}}_{\mathbf{f}}(\mathbf{A},\mathbf{a},\mathbf{g})$
is nonempty and vaguely compact.
\end{theorem}

\begin{remark} If $I^-$ is nonempty, then condition (a) follows
immediately from~(\ref{s}) and, hence, it can be omitted. It also
holds automatically if the space $\mathrm X$ is countable at
infinity (e.g., for $\mathrm X=\mathbb R^n$).\end{remark}

\begin{remark}Regarding condition (c), note that a compact set $K\subset\mathrm X$
might be of infinite capacity; $C(K)$ is necessarily finite provided
the kernel is strictly positive definite~\cite{F1}. On the other
hand, even for the Newtonian kernel, sets of finite capacity might
be noncompact~\cite{L}.\end{remark}

\begin{remark}Condition (c) is essential for the validity of Theorem~\ref{exist}. See Sec.~\ref{sharp} for
some examples, illustrating its sharpness.
\end{remark}

\begin{corollary}\label{cor:exist} If $\mathbf{A}=\mathbf{K}$ is compact, then,
for any $\boldsymbol{\sigma}$, $\mathbf{f}$, $\mathbf{g}$, and\/
$\mathbf{a}$, $\mathfrak
S^{\boldsymbol{\sigma}}_{\mathbf{f}}(\mathbf{A},\mathbf{a},\mathbf{g})$
is nonempty and vaguely compact.
\end{corollary}

\begin{proof}This is an immediate consequence of Theorem~\ref{exist},
since $g_i$ is bounded on~$K_i$.\end{proof}

\subsection{On continuity of $G^{\boldsymbol{\sigma}}_{\mathbf{f}}(\mathbf{A},\mathbf{a},\mathbf{g})$
and $\boldsymbol{\lambda}^{\boldsymbol{\sigma}}_{\mathbf{A}}$ with
respect to $(\mathbf{A},\boldsymbol{\sigma})$}\label{appr1}

We write $\mathbf{A}_s\downarrow\mathbf{A}$, where
$\mathbf{A}_s=(A_i^s)_{i\in I}$, $s\in S$, is a net of condensers,
if $\mathbf{A}_{s_2}\leqslant\mathbf{A}_{s_1}$ whenever
$s_1\leqslant s_2$ and
\[\bigcap_{s\in S}\,A_i^s=A_i\quad\text{for all \ } i\in I.\]

\begin{theorem}\label{cor:sigma} Let\/
$\mathbf{A}_s\downarrow\mathbf{A}$, and let for some $s_0\in S$ all
the assumptions\/\footnote{Including the standing ones.}
of~Theorem~{\rm\ref{exist}} with~$\mathbf{A}_{s_0}$ instead
of~$\mathbf{A}$ be satisfied. Let $\boldsymbol{\sigma}_s$ be a
constraint associated with~$\mathbf{A}_s$, and let
$(\boldsymbol{\sigma}_s)_{s\in S}$ decrease and converge vaguely
to~$\boldsymbol{\sigma}$. Then
\begin{equation*}\label{sigma}
G^{\boldsymbol{\sigma}}_{\mathbf{f}}(\mathbf{A},\mathbf{a},\mathbf{g})=\lim_{s\in
S}\,G^{\boldsymbol{\sigma}_s}_{\mathbf{f}}(\mathbf{A}_s,\mathbf{a},\mathbf{g}).\end{equation*}
Fix arbitrary
$\boldsymbol{\lambda}^{\boldsymbol{\sigma}_{s}}_{\mathbf{A}_{s}}\in\mathfrak
S^{\boldsymbol{\sigma}_s}_{\mathbf{f}}(\mathbf{A}_s,\mathbf{a},\mathbf{g})$,
where $s\geqslant s_0$, and\/
$\boldsymbol{\lambda}^{\boldsymbol{\sigma}}_{\mathbf{A}}\in\mathfrak
S^{\boldsymbol{\sigma}}_{\mathbf{f}}(\mathbf{A},\mathbf{a},\mathbf{g})$~---
such minimizers exist. Then every vague cluster point of the net\/
$(\boldsymbol{\lambda}^{\boldsymbol{\sigma}_{s}}_{\mathbf{A}_{s}})_{s\in
S}$ is an element of\/ $\mathfrak
S^{\boldsymbol{\sigma}}_{\mathbf{f}}(\mathbf{A},\mathbf{a},\mathbf{g})$.
Furthermore,
$\boldsymbol{\lambda}^{\boldsymbol{\sigma}_{s}}_{\mathbf{A}_{s}}\to
\boldsymbol{\lambda}^{\boldsymbol{\sigma}}_{\mathbf{A}}$ strongly,
i.e.
\begin{equation*}\label{sigma1}\lim_{s\in S}\,\|\boldsymbol{\lambda}^{\boldsymbol{\sigma}_{s}}_{\mathbf{A}_{s}}-
\boldsymbol{\lambda}^{\boldsymbol{\sigma}}_{\mathbf{A}}\|_{\mathcal
E^+(\mathbf{A}_{s_0})}=0.
\end{equation*}
\end{theorem}

\begin{corollary}\label{contcor1}Under the assumptions of Theorem~{\rm\ref{cor:sigma}}, if,
moreover, the kernel~$\kappa$ is strictly positive definite\/
{\rm(}hence, perfect\/{\rm)} and all~$A_i^{s_0}$, $i\in I$, are
mutually disjoint, then the\/ {\rm(}unique\/{\rm)} minimizer
$\boldsymbol{\lambda}^{\boldsymbol{\sigma}_{s}}_{\mathbf{A}_{s}}\in\mathfrak
S^{\boldsymbol{\sigma}_s}_{\mathbf{f}}(\mathbf{A}_s,\mathbf{a},\mathbf{g})$,
where $s\geqslant s_0$, approaches the\/ {\rm(}unique\/{\rm)}
minimizer
$\boldsymbol{\lambda}^{\boldsymbol{\sigma}}_{\mathbf{A}}\in\mathfrak
S^{\boldsymbol{\sigma}}_{\mathbf{f}}(\mathbf{A},\mathbf{a},\mathbf{g})$
both vaguely and strongly.
\end{corollary}

In the rest of Sec.~\ref{appr1}, let $\mathbf{A}$ and $\mathbf{g}$
satisfy all the conditions of Theorem~{\rm\ref{exist}}. We proceed
by analyzing continuity properties of
$G^{\boldsymbol{\sigma}}_{\mathbf{f}}(\mathbf{A},\mathbf{a},\mathbf{g})$
and~$\boldsymbol{\lambda}^{\boldsymbol{\sigma}}_{\mathbf{A}}$ under
exhaustion of~$\mathbf{A}$ by compact~$\mathbf{K}$.

\begin{theorem}\label{cor:cont} There exists a net
$(\beta^*_{\mathbf{K}})_{\mathbf{K}\in
\{\mathbf{K}\}_{\mathbf{A}}}\subset(1,\infty)$ decreasing to~$1$ and
such that, for any $\beta_{\mathbf{K}}\in[1,\beta^*_{\mathbf{K}}]$,
\begin{equation}
G^{\boldsymbol{\sigma}}_{\mathbf{f}}(\mathbf{A},\mathbf{a},\mathbf{g})=\lim_{\mathbf{K}\uparrow\mathbf{A}}\,
G^{\beta_{\mathbf{K}}\boldsymbol{\sigma}_{\mathbf{K}}}_{\mathbf{f}}(\mathbf{K},\mathbf{a},\mathbf{g}),\label{contnew}
\end{equation}
where $\boldsymbol{\sigma}_{\mathbf{K}}:=(\sigma^i_{K_i})_{i\in I}$.
Fix arbitrary
$\boldsymbol{\lambda}^{\beta_{\mathbf{K}}\boldsymbol{\sigma}_{\mathbf{K}}}_{\mathbf{K}}\in\mathfrak
S^{\beta_{\mathbf{K}}\boldsymbol{\sigma}_{\mathbf{K}}}_{\mathbf{f}}(\mathbf{K},\mathbf{a},\mathbf{g})$,
where $\mathbf{K}\in\{\mathbf{K}\}_{\mathbf{A}}$ is sufficiently
large, and
$\boldsymbol{\lambda}^{\boldsymbol{\sigma}}_{\mathbf{A}}\in\mathfrak
S^{\boldsymbol{\sigma}}_{\mathbf{f}}(\mathbf{A},\mathbf{a},\mathbf{g})$~---
such minimizers exist. Then every vague cluster point of the net
$(\boldsymbol{\lambda}^{\beta_{\mathbf{K}}\boldsymbol{\sigma}_{\mathbf{K}}}_{\mathbf{K}})_{\mathbf{K}\in
\{\mathbf{K}\}_{\mathbf{A}}}$ is an element of $\mathfrak
S^{\boldsymbol{\sigma}}_{\mathbf{f}}(\mathbf{A},\mathbf{a},\mathbf{g})$.
Furthermore,
$\boldsymbol{\lambda}^{\beta_{\mathbf{K}}\boldsymbol{\sigma}_{\mathbf{K}}}_{\mathbf{K}}\to
\boldsymbol{\lambda}^{\boldsymbol{\sigma}}_{\mathbf{A}}$ strongly,
i.e.
\begin{equation*}\label{sigma2}\lim_{\mathbf{K}\uparrow\mathbf{A}}\,
\|\boldsymbol{\lambda}^{\beta_{\mathbf{K}}\boldsymbol{\sigma}_{\mathbf{K}}}_{\mathbf{K}}-
\boldsymbol{\lambda}^{\boldsymbol{\sigma}}_{\mathbf{A}}\|_{\mathcal
E^+(\mathbf A)}=0.
\end{equation*}
\end{theorem}

\begin{corollary}\label{contcor2}With the notation of Theorem~{\rm\ref{cor:cont}}, if
the kernel~$\kappa$ is strictly positive definite\/ {\rm(}hence,
perfect\/{\rm)} and all~$A_i$, $i\in I$, are mutually disjoint, then
the\/ {\rm(}unique\/{\rm)} minimizer
$\boldsymbol{\lambda}^{\beta_{\mathbf{K}}\boldsymbol{\sigma}_{\mathbf{K}}}_{\mathbf{K}}\in\mathfrak
S^{\beta_{\mathbf{K}}\boldsymbol{\sigma}_{\mathbf{K}}}_{\mathbf{f}}(\mathbf{K},\mathbf{a},\mathbf{g})$,
where $\mathbf{K}\in \{\mathbf{K}\}_{\mathbf{A}}$, approaches the\/
{\rm(}unique\/{\rm)} minimizer
$\boldsymbol{\lambda}^{\boldsymbol{\sigma}}_{\mathbf{A}}\in\mathfrak
S^{\boldsymbol{\sigma}}_{\mathbf{f}}(\mathbf{A},\mathbf{a},\mathbf{g})$
both vaguely and strongly.
\end{corollary}

\begin{remark}
The value
$G^{\boldsymbol{\sigma}}_{\mathbf{f}}(\mathbf{A},\mathbf{a},\mathbf{g})$
remains unchanged if $\mathcal
E^+_{\boldsymbol{\sigma},\mathbf{f}}(\mathbf{A},\mathbf{a},\mathbf{g})$
in its definition is replaced by the class of all
$\boldsymbol{\mu}\in\mathcal
E^+_{\boldsymbol{\sigma},\mathbf{f}}(\mathbf{A},\mathbf{a},\mathbf{g})$
such that ${\rm supp}\,\mu^i$, $i\in I$, are {\em compact\/}.
Indeed, this is concluded from (\ref{contnew}) with
$\beta_{\mathbf{K}}=1$ for all
$\mathbf{K}\in\{\mathbf{K}\}_{\mathbf{A}}$.\end{remark}

The proofs of Theorems~\ref{exist}, \ref{cor:sigma}
and~\ref{cor:cont}, to be given in Sections~\ref{sec:proof.th.str},
\ref{sec:proof.cor:sigma} and~\ref{sec:proof.th.cont} below (see
also Sec.~\ref{sec:extremal} for some crucial auxiliary notions and
results), are based on a theorem on the strong completeness of
proper subspaces of the semimetric space~$\mathcal E^+(\mathbf{A})$,
which is a subject of Sec.~\ref{sec:11}.

\section{Strong completeness of classes of vector measures}\label{sec:11}

Recall that we are working under the standing assumptions, stated
in~Sec.~\ref{sec:standing}. Write
\begin{equation*} \mathfrak
M^+(\mathbf{A},\leqslant\!\mathbf{a},\mathbf{g}):=\bigl\{\boldsymbol{\mu}\in\mathfrak
M^+(\mathbf{A}):\quad\langle g_i,\mu^i\rangle\leqslant a_i\text{ \
for all \ } i\in I\bigr\},
\end{equation*}
\[\mathfrak
M^+_{\boldsymbol{\sigma}}(\mathbf{A},\leqslant\!\mathbf{a},\mathbf{g}):=\mathfrak
M^+(\mathbf{A},\leqslant\!\mathbf{a},\mathbf{g})\cap\mathfrak
M^+_{\boldsymbol{\sigma}}(\mathbf{A}),\]
\[\mathcal
E^+(\mathbf{A},\leqslant\!\mathbf{a},\mathbf{g}):=\mathfrak
M^+(\mathbf{A},\leqslant\!\mathbf{a},\mathbf{g})\cap\mathcal
E^+(\mathbf{A}),\] and \[\mathcal
E^+_{\boldsymbol{\sigma}}(\mathbf{A},\leqslant\!\mathbf{a},\mathbf{g}):=\mathcal
E^+(\mathbf{A},\leqslant\!\mathbf{a},\mathbf{g})\cap\mathfrak
M^+_{\boldsymbol{\sigma}}(\mathbf{A}).\] Our next purpose is to show
that $\mathcal E^+(\mathbf{A},\leqslant\!\mathbf{a},\mathbf{g})$ and
$\mathcal
E^+_{\boldsymbol{\sigma}}(\mathbf{A},\leqslant\!\mathbf{a},\mathbf{g})$,
treated as topological subspaces of the semimetric space $\mathcal
E^+(\mathbf{A})$, are {\em strongly complete\/}.

\subsection{Auxiliary assertions}

\begin{lemma}\label{Mbounded}
$\mathfrak M^+(\mathbf{A},\leqslant\!\mathbf{a},\mathbf{g})$ and
$\mathfrak
M^+_{\boldsymbol{\sigma}}(\mathbf{A},\leqslant\!\mathbf{a},\mathbf{g})$
are vaguely bounded and, hence, they are vaguely compact.\end{lemma}

\begin{proof} Fix $i\in I$, and let a compact set $K_i\subset A_i$ be given. Since
$g_i$ is positive and continuous, the relation
\[
a_i\geqslant\langle g_i,\mu^i\rangle\geqslant\mu^i(K_i)\,\min_{x\in
K_i}\,g_i(x),\quad\text{where \ }\mu\in\mathfrak
M^+(\mathbf{A},\leqslant\!\mathbf{a},\mathbf{g}),\] yields
\[\sup_{\boldsymbol{\mu}\in\mathfrak M^+(\mathbf{A},\leqslant\mathbf{a},\mathbf{g})}\,\mu^i(K_i)<\infty.\]
This implies that $\mathfrak
M^+(\mathbf{A},\leqslant\!\mathbf{a},\mathbf{g})$ is vaguely
bounded; hence, by Lemma~\ref{lem:vaguecomp}, it is vaguely
relatively compact. In fact, it is vaguely compact, since it is
vaguely closed in consequence of Lemma~\ref{lemma:lower} with
$\mathrm Y=A_i$ and $\psi=g_i$. Having observed that also $\mathfrak
M^+_{\boldsymbol{\sigma}}(\mathbf{A})$ is vaguely closed, we then
conclude that $\mathfrak
M^+_{\boldsymbol{\sigma}}(\mathbf{A},\leqslant\!\mathbf{a},\mathbf{g})$
is vaguely compact as well, which completes the proof.\end{proof}

\begin{lemma}\label{lem:aux2} If a net $(\boldsymbol{\mu}_s)_{s\in S}\subset\mathcal
E^+(\mathbf{A},\leqslant\!\mathbf{a},\mathbf{g})$ is strongly
bounded, then every its vague cluster point~$\boldsymbol{\mu}$ has
finite energy.
\end{lemma}

\begin{proof}
Note that, by~(\ref{Ren}), the net of scalar measures
$(R\boldsymbol{\mu}_s)_{s\in S}\subset\mathcal E$ is strongly
bounded as well. We proceed by showing that so are
$(R\boldsymbol{\mu}_s^+)_{s\in S}$ and
$(R\boldsymbol{\mu}_s^-)_{s\in S}$, i.e.,
\begin{equation}
\sup_{s\in S}\,\|R\boldsymbol{\mu}_s^\pm\|^2<\infty.  \label{7.1}
\end{equation}
Of course, this needs to be verified only when $I^-\ne\varnothing$;
then, according to the standing assumptions, (\ref{s})
and~(\ref{bou}) hold. Since $\langle g_i,\mu_s^i\rangle\leqslant
a_i$, we conclude that $\mu_s^i(\mathrm X)\leqslant
a_ig_{i,\inf}^{-1}$ for all $i\in I$ and $s\in S$. Therefore, by
(\ref{s}),
\[\sup_{s\in S}\,R\boldsymbol{\mu}_s^\pm(\mathrm X)\leqslant
\sum_{i\in I}\,a_ig_{i,\inf}^{-1}<\infty.\] Because of~(\ref{bou}),
this implies that
$\kappa(R\boldsymbol{\mu}^+_s,R\boldsymbol{\mu}^-_s)$ remains
bounded from above on~$S$; hence, so do
$\|R\boldsymbol{\mu}^+_s\|^2$ and $\|R\boldsymbol{\mu}^-_s\|^2$.

If $(\boldsymbol{\mu}_d)_{d\in D}$ is a subnet of
$(\boldsymbol{\mu}_s)_{s\in S}$ that converges vaguely
to~$\boldsymbol{\mu}$, then $(R\boldsymbol{\mu}^+_d)_{d\in D}$ and
$(R\boldsymbol{\mu}^-_d)_{d\in D}$ converge vaguely
to~$R\boldsymbol{\mu}^+$ and~$R\boldsymbol{\mu}^-$, respectively.
Using the fact that the map $(\nu_1,\nu_2)\mapsto\nu_1\otimes\nu_2$
from $\mathfrak M^+(\mathrm X)\times\mathfrak M^+(\mathrm X)$ into
$\mathfrak M^+(\mathrm X\times\mathrm X)$ is vaguely continuous
(see~\cite[Chap.~3, \S~5, exerc.~5]{B2}) and applying
Lemma~\ref{lemma:lower} to $\mathrm Y=\mathrm X\times\mathrm X$ and
$\psi=\kappa$, we conclude from~(\ref{7.1}) that
$R\boldsymbol{\mu}^+$ and $R\boldsymbol{\mu}^-$ are both of finite
energy. By Corollary~\ref{unknown}, this means
$\boldsymbol{\mu}\in\mathcal E^+(\mathbf A)$, as was to be
proved.\end{proof}

\begin{corollary}\label{cor:aux1} If a net $(\boldsymbol{\mu}_s)_{s\in
S}\subset\mathcal E^+(\mathbf{A},\leqslant\!\mathbf{a},\mathbf{g})$
is strongly bounded, then
\begin{equation} \sup_{s\in S}\,\|\mu_s^i\|^2<\infty,\quad i\in I.\label{7.1i}
\end{equation}
\end{corollary}

\begin{proof} As an application of Lemma~\ref{casei}, we obtain
\begin{equation*}\inf_{s\in S}\,\sum_{i\ne
j,\ i,j\in
I^\pm}\,\langle\kappa,\mu_s^i\otimes\mu_s^j\rangle>-\infty.\label{bdbl}\end{equation*}
When combined with (\ref{vecen}) and~(\ref{7.1}), this yields the
corollary.\end{proof}

\subsection{Strong completeness of
$\mathcal E^+(\mathbf{A},\leqslant\!\mathbf{a},\mathbf{g})$ and
$\mathcal
E^+_{\boldsymbol{\sigma}}(\mathbf{A},\leqslant\!\mathbf{a},\mathbf{g})$}\label{sec:strong}

\begin{theorem}\label{th:strong} The following assertions
hold:\smallskip
\begin{itemize}
\item[\rm(i)] The
semimetric space $\mathcal
E^+(\mathbf{A},\leqslant\!\mathbf{a},\mathbf{g})$ is complete. In
more detail, if $(\boldsymbol{\mu}_s)_{s\in S}$ is a strong Cauchy
net in $\mathcal E^+(\mathbf{A},\leqslant\!\mathbf{a},\mathbf{g})$
and $\boldsymbol{\mu}$ is one of its vague cluster points\/
{\rm(}such a~$\boldsymbol{\mu}$ exists\/{\rm)}, then
$\boldsymbol{\mu}\in\mathcal
E^+(\mathbf{A},\leqslant\!\mathbf{a},\mathbf{g})$ and
$\boldsymbol{\mu}_s\to\boldsymbol{\mu}$ strongly, i.e.
\begin{equation}
\lim_{s\in S}\,\|\boldsymbol{\mu}_s-\boldsymbol{\mu}\|_{\mathcal
E^+(\mathbf A)}=0.\label{str}
\end{equation}
\item[\rm(ii)] If the kernel $\kappa$ is strictly positive
definite while all $A_i$, $i\in I$, are mutually disjoint, then the
strong topology on $\mathcal
E^+(\mathbf{A},\leqslant\!\mathbf{a},\mathbf{g})$ is finer than the
vague one.  In more detail, if $(\boldsymbol{\mu}_s)_{s\in
S}\subset\mathcal E^+(\mathbf{A},\leqslant\!\mathbf{a},\mathbf{g})$
converges strongly to $\boldsymbol{\mu}_0\in\mathcal
E^+(\mathbf{A})$, then actually $\boldsymbol{\mu}_0\in\mathcal
E^+(\mathbf{A},\leqslant\!\mathbf{a},\mathbf{g})$ and
$\boldsymbol{\mu}_s\to\boldsymbol{\mu}_0$ vaguely.\smallskip
\item[\rm(iii)] Both the assertions\/ {\rm(i)} and\/ {\rm(ii)} remain valid if\/ $\mathcal
E^+(\mathbf{A},\leqslant\!\mathbf{a},\mathbf{g})$ is  replaced
everywhere by $\mathcal
E^+_{\boldsymbol{\sigma}}(\mathbf{A},\leqslant\!\mathbf{a},\mathbf{g})$.\end{itemize}\end{theorem}

\begin{proof}To verify (i), fix a strong Cauchy net  $(\boldsymbol{\mu}_s)_{s\in S}\subset\mathcal
E^+(\mathbf{A},\leqslant\!\mathbf{a},\mathbf{g})$. Since such a net
converges strongly to any of its strong cluster points,
$(\boldsymbol{\mu}_s)_{s\in S}$ can be assumed to be strongly
bounded. Then, by Lemmas~\ref{Mbounded} and~\ref{lem:aux2}, a vague
cluster point~$\boldsymbol{\mu}$ of~$(\boldsymbol{\mu}_s)_{s\in S}$
exists and, moreover,
\begin{equation}
\boldsymbol{\mu}\in\mathcal
E^+(\mathbf{A},\leqslant\!\mathbf{a},\mathbf{g}). \label{leqslant1}
\end{equation}
We next proceed by proving (\ref{str}). Of course, there is no loss
of generality in assuming $(\boldsymbol{\mu}_s)_{s\in S}$ to
converge vaguely to~$\boldsymbol{\mu}$. Then, by
Lemma~\ref{lem:vague'}, $(R\boldsymbol{\mu}^+_s)_{s\in S}$ and
$(R\boldsymbol{\mu}^-_s)_{s\in S}$ converge vaguely
to~$R\boldsymbol{\mu}^+$ and~$R\boldsymbol{\mu}^-$, respectively.
Since, by~(\ref{7.1}), these nets are strongly bounded in~$\mathcal
E^+$, the property~(C$_2$) (see~Sec.~\ref{sec:2}) shows that they
approach~$R\boldsymbol{\mu}^+$ and~$R\boldsymbol{\mu}^-$,
respectively, in the weak topology as well, and so
$R\boldsymbol{\mu}_s\to R\boldsymbol{\mu}$ weakly. This gives,
by~(\ref{seminorm}),
\[
\|\boldsymbol{\mu}_s-\boldsymbol{\mu}\|^2=\|R\boldsymbol{\mu}_s-R\boldsymbol{\mu}\|^2=\lim_{\ell\in
S}\,\kappa(R\boldsymbol{\mu}_s-R\boldsymbol{\mu},R\boldsymbol{\mu}_s-R\boldsymbol{\mu}_\ell)
\]
and hence, by the Cauchy--Schwarz inequality,
\[
\|\boldsymbol{\mu}_s-\boldsymbol{\mu}\|^2\leqslant
\|\boldsymbol{\mu}_s-\boldsymbol{\mu}\|\,\liminf_{\ell\in
S}\,\|\boldsymbol{\mu}_s-\boldsymbol{\mu}_\ell\|,
\] which proves
(\ref{str}) as required, because
$\|\boldsymbol{\mu}_s-\boldsymbol{\mu}_\ell\|$ becomes arbitrarily
small when $s,\,\ell\in S$ are sufficiently large. The proof of (i)
is complete.

To establish (ii), suppose now that the kernel $\kappa$ is strictly
positive definite, while all $A_i$, $i\in I$, are mutually disjoint,
and let the net $(\boldsymbol{\mu}_s)_{s\in S}$ converge strongly to
some $\boldsymbol{\mu}_0\in\mathcal E^+(\mathbf{A})$. Given a vague
limit point~$\boldsymbol{\mu}$ of~$(\boldsymbol{\mu}_s)_{s\in S}$,
we conclude from~(\ref{str}) that
$\|\boldsymbol{\mu}_0-\boldsymbol{\mu}\|=0$, hence
$R\boldsymbol{\mu}_0=R\boldsymbol{\mu}$ since $\kappa$ is strictly
positive definite, and finally $\boldsymbol{\mu}_0=\boldsymbol{\mu}$
because $A_i$, $i\in I$, are mutually disjoint. In view
of~(\ref{leqslant1}), this means that $\boldsymbol{\mu}_0\in\mathcal
E^+(\mathbf{A},\leqslant\!\mathbf{a},\mathbf{g})$, which is a part
of the desired conclusion. Moreover, $\boldsymbol{\mu}_0$ has thus
been shown to be identical to any vague cluster point
of~$(\boldsymbol{\mu}_s)_{s\in S}$. Since the vague topology is
Hausdorff, this implies that $\boldsymbol{\mu}_0$ is actually the
vague limit of~$(\boldsymbol{\mu}_s)_{s\in S}$ (cf.~\cite[Chap.~I,
\S~9, n$^\circ$\,1, cor.]{B1}), as claimed.

Finally, we observe that all the arguments applied above remain
valid if $\mathcal E^+(\mathbf{A},\leqslant\!\mathbf{a},\mathbf{g})$
is replaced everywhere by $\mathcal
E^+_{\boldsymbol{\sigma}}(\mathbf{A},\leqslant\!\mathbf{a},\mathbf{g})$.
This yields~(iii).\end{proof}

\begin{remark} Since
the semimetric spaces $\mathcal
E^+(\mathbf{A},\leqslant\!\mathbf{a},\mathbf{g})$ and $\mathcal
E^+_{\boldsymbol{\sigma}}(\mathbf{A},\leqslant\!\mathbf{a},\mathbf{g})$
are isometric to their $R$-images, Theorem~\ref{th:strong} has thus
singled out  {\em strongly complete\/} topological subspaces of the
pre-Hilbert space~$\mathcal E$, whose elements are {\em signed\/}
measures. This is of a independent interest because, according to a
well-known counterexample by H.~Cartan~\cite{Car}, the whole
space~$\mathcal E$ is strongly incomplete even for the Newtonian
kernel $|x-y|^{2-n}$ in~$\mathbb R^n$, $n\geqslant3$.\end{remark}

\section{Extremal measures in the constrained energy problem}\label{sec:extremal}
To apply Theorem~\ref{th:strong} to the constrained energy problem,
we next proceed by introducing the concept of extremal measure
defined as a strong and, simultaneously, a vague limit point of a
minimizing net. See below for strict definitions and related
auxiliary results.

\subsection{Extremal measures: existence, uniqueness, and vague compactness}

\begin{definition}We call a net $(\boldsymbol{\mu}_s)_{s\in S}$ {\em minimizing\/} if
$(\boldsymbol{\mu}_s)_{s\in S}\subset\mathcal
E^+_{\boldsymbol{\sigma},\mathbf{f}}(\mathbf{A},\mathbf{a},\mathbf{g})$
and
\begin{equation}\lim_{s\in S}\,G_{\mathbf{f}}(\boldsymbol{\mu}_s)=
G^{\boldsymbol{\sigma}}_{\mathbf{f}}(\mathbf{A},\mathbf{a},\mathbf{g}).
\label{min}\end{equation}\end{definition}

Let $\mathbb
M^{\boldsymbol{\sigma}}_{\mathbf{f}}(\mathbf{A},\mathbf{a},\mathbf{g})$
consist of all minimizing nets, and let $\mathcal
M^{\boldsymbol{\sigma}}_{\mathbf{f}}(\mathbf{A},\mathbf{a},\mathbf{g})$
be the union of the vague cluster sets of
$(\boldsymbol{\mu}_s)_{s\in S}$, where $(\boldsymbol{\mu}_s)_{s\in
S}$ ranges over $\mathbb
M^{\boldsymbol{\sigma}}_{\mathbf{f}}(\mathbf{A},\mathbf{a},\mathbf{g})$.

\begin{definition}\label{def:extr} We call $\boldsymbol{\gamma}\in\mathcal E^+(\mathbf{A})$ {\em extremal\/}
if there exists $(\boldsymbol{\mu}_s)_{s\in S}\in\mathbb
M^{\boldsymbol{\sigma}}_{\mathbf{f}}(\mathbf{A},\mathbf{a},\mathbf{g})$
that converges to~$\boldsymbol{\gamma}$ both strongly and vaguely;
such a net $(\boldsymbol{\mu}_s)_{s\in S}$ is said to {\it
generate\/}~$\boldsymbol{\gamma}$. The class of all extremal
measures will be denoted by $\mathfrak
E^{\boldsymbol{\sigma}}_{\mathbf{f}}(\mathbf{A},\mathbf{a},\mathbf{g})$.
\end{definition}

\begin{lemma}\label{lemma:WM} The following assertions hold
true:\smallskip
\begin{itemize}
\item[\rm(i)] From every minimizing net one can select a subnet generating an extremal measure;
hence, $\mathfrak
E^{\boldsymbol{\sigma}}_{\mathbf{f}}(\mathbf{A},\mathbf{a},\mathbf{g})$
is nonempty. Furthermore,
\begin{equation}\label{ME}\mathfrak E^{\boldsymbol{\sigma}}_{\mathbf{f}}(\mathbf{A},\mathbf{a},\mathbf{g})
\subset\mathcal
E^+_{\boldsymbol{\sigma}}(\mathbf{A},\leqslant\!\mathbf{a},\mathbf{g})\end{equation}
and
\begin{equation}\mathfrak E^{\boldsymbol{\sigma}}_{\mathbf{f}}(\mathbf{A},\mathbf{a},\mathbf{g})=
\mathcal
M^{\boldsymbol{\sigma}}_{\mathbf{f}}(\mathbf{A},\mathbf{a},\mathbf{g}).
\label{WM}\end{equation}
\item[\rm(ii)] Every minimizing
net converges strongly to every extremal measure; hence, $\mathfrak
E^{\boldsymbol{\sigma}}_{\mathbf{f}}(\mathbf{A},\mathbf{a},\mathbf{g})$
is contained in an equivalence class in~$\mathcal
E^+(\mathbf{A})$.\smallskip
\item[\rm(iii)] The class $\mathfrak E^{\boldsymbol{\sigma}}_{\mathbf{f}}(\mathbf{A},\mathbf{a},\mathbf{g})$ is vaguely compact.
\end{itemize}
\end{lemma}

\begin{proof}Fix $(\boldsymbol{\mu}_s)_{s\in S}$ and $(\boldsymbol{\nu}_t)_{t\in T}$ in $\mathbb
M^{\boldsymbol{\sigma}}_{\mathbf{f}}(\mathbf{A},\mathbf{a},\mathbf{g})$.
Then
\begin{equation}
\lim_{(s,t)\in S\times
T}\,\|\boldsymbol{\mu}_s-\boldsymbol{\nu}_t\|^2=0, \label{fund}
\end{equation}
where $S\times T$ is the directed product (see, e.g.,
\cite[Chap.~2,~\S~3]{K}) of the directed sets~$S$ and~$T$. Indeed,
in the same manner as in the proof of~Lemma~\ref{lemma:unique} we
get
\[0\leqslant\|R\boldsymbol{\mu}_s-R\boldsymbol{\nu}_t\|^2\leqslant-
4G^{\boldsymbol{\sigma}}_{\mathbf{f}}(\mathbf{A},\mathbf{a},\mathbf{g})+
2G_{\mathbf{f}}(\boldsymbol{\mu}_s)+2G_{\mathbf{f}}(\boldsymbol{\nu}_t),\]
which yields (\ref{fund}) when combined with (\ref{min}).

Relation~(\ref{fund}) implies that $(\boldsymbol{\mu}_s)_{s\in S}$
is strongly fundamental. Therefore, by Theorem~\ref{th:strong},
(iii), for every vague cluster point~$\boldsymbol{\mu}$
of~$(\boldsymbol{\mu}_s)_{s\in S}$ (such a~$\boldsymbol{\mu}$
exists) we have $\boldsymbol{\mu}\in\mathcal
E^+_{\boldsymbol{\sigma}}(\mathbf{A},\leqslant\!\mathbf{a},\mathbf{g})$
and $\boldsymbol{\mu}_s\to\boldsymbol{\mu}$ strongly. Thus,
$\boldsymbol{\mu}$ is an extremal measure; actually,
\[\mathcal
M^{\boldsymbol{\sigma}}_{\mathbf{f}}(\mathbf{A},\mathbf{a},\mathbf{g})\subset\mathfrak
E^{\boldsymbol{\sigma}}_{\mathbf{f}}(\mathbf{A},\mathbf{a},\mathbf{g}).\]
Since the inverse inclusion is obvious, relations~(\ref{ME})
and~(\ref{WM}) follow.

Having thus proved (i), we proceed by verifying (ii). Fix arbitrary
$(\boldsymbol{\mu}_s)_{s\in S}\in\mathbb
M^{\boldsymbol{\sigma}}_{\mathbf{f}}(\mathbf{A},\mathbf{a},\mathbf{g})$
and $\boldsymbol{\gamma}\in\mathfrak
E^{\boldsymbol{\sigma}}_{\mathbf{f}}(\mathbf{A},\mathbf{a},\mathbf{g})$.
Then, according to Definition~\ref{def:extr}, there exists a net
in~$\mathbb
M^{\boldsymbol{\sigma}}_{\mathbf{f}}(\mathbf{A},\mathbf{a},\mathbf{g})$,
say $(\boldsymbol{\nu}_t)_{t\in T}$, that converges
to~$\boldsymbol{\gamma}$ strongly. Repeated application
of~(\ref{fund}) shows that also $(\boldsymbol{\mu}_s)_{s\in S}$
converges to~$\boldsymbol{\gamma}$ strongly, as claimed.

To establish (iii), it is enough to prove that $\mathcal
M^{\boldsymbol{\sigma}}_{\mathbf{f}}(\mathbf{A},\mathbf{a},\mathbf{g})$
is vaguely compact. Fix $(\boldsymbol{\gamma}_s)_{s\in
S}\subset\mathcal
M^{\boldsymbol{\sigma}}_{\mathbf{f}}(\mathbf{A},\mathbf{a},\mathbf{g})$.
It follows from~(\ref{ME}) and~Lemma~\ref{Mbounded} that there
exists a vague cluster point~$\boldsymbol{\gamma}_0$ of
$(\boldsymbol{\gamma}_s)_{s\in S}$; let
$(\boldsymbol{\gamma}_t)_{t\in T}$ be a subnet
of~$(\boldsymbol{\gamma}_s)_{s\in S}$ that converges vaguely
to~$\boldsymbol{\gamma}_0$. Then for every $t\in T$ one can choose
$(\boldsymbol{\mu}_{s_t})_{s_t\in S_t}\in\mathbb
M^{\boldsymbol{\sigma}}_{\mathbf{f}}(\mathbf{A},\mathbf{a},\mathbf{g})$
converging vaguely to~$\boldsymbol{\gamma}_t$. Consider the
Cartesian product $\prod\,\{S_t: t\in T\}$~--- that is, the
collection of all functions~$\beta$ on~$T$ with $\beta(t)\in S_t$,
and let~$D$ denote the directed product $T\times\prod\,\{S_t: t\in
T\}$. Given $(t,\beta)\in D$, write
$\boldsymbol{\mu}_{(t,\beta)}:=\boldsymbol{\mu}_{\beta(t)}$. Then
the theorem on iterated limits from \cite[Chap.~2, \S~4]{K} yields
that the net $(\boldsymbol{\mu}_{(t,\beta)})_{(t,\beta)\in D}$
belongs to $\mathbb
M^{\boldsymbol{\sigma}}_{\mathbf{f}}(\mathbf{A},\mathbf{a},\mathbf{g})$
and converges vaguely to~$\boldsymbol{\gamma}_0$. Thus,
$\boldsymbol{\gamma}_0\in\mathcal
M^{\boldsymbol{\sigma}}_{\mathbf{f}}(\mathbf{A},\mathbf{a},\mathbf{g})$.\end{proof}

\begin{corollary}\label{IorII}If Case II takes place, then
\begin{equation}\label{conclII}G_{\mathbf{f}}(\boldsymbol{\gamma})=G^{\boldsymbol{\sigma}}_{\mathbf{f}}(\mathbf{A},\mathbf{a},\mathbf{g})\quad
\text{for all \ }\boldsymbol{\gamma}\in\mathfrak
E^{\boldsymbol{\sigma}}_{\mathbf{f}}(\mathbf{A},\mathbf{a},\mathbf{g}).\end{equation}
\end{corollary}

\begin{proof} Applying (\ref{yusss}) to $\boldsymbol{\mu}_s$, $s\in S$, and
$\boldsymbol{\gamma}$, where $(\boldsymbol{\mu}_s)_{s\in
S}\in\mathbb
M^{\boldsymbol{\sigma}}_{\mathbf{f}}(\mathbf{A},\mathbf{a},\mathbf{g})$
and $\boldsymbol{\gamma}\in\mathfrak
E^{\boldsymbol{\sigma}}_{\mathbf{f}}(\mathbf{A},\mathbf{a},\mathbf{g})$
are arbitrarily given, and using the fact that, in accordance with
Lemma~\ref{lemma:WM}, $\boldsymbol{\mu}_s\to\boldsymbol{\gamma}$
strongly, we get
\[G_{\mathbf{f}}(\boldsymbol{\gamma})=\|R\boldsymbol{\gamma}+\zeta\|^2-\|\zeta\|^2=\lim_{s\in S}\,
\bigl[\|R\boldsymbol{\mu}_s+\zeta\|^2-\|\zeta\|^2\bigr]=\lim_{s\in
S}\,G_{\mathbf{f}}(\boldsymbol{\mu}_s).\] Substituting (\ref{min})
into the last relation gives (\ref{conclII}), as desired.\end{proof}

\subsection{Extremal measures: $g_i$-masses of the $i$-components}

\begin{lemma}\label{lemma:exist'} Let $\kappa$, $\mathbf{A}$,
$\mathbf{a}$, and $\mathbf{g}$ satisfy all the assumptions
\mbox{\rm(a)--(c)} of Theorem\/~{\rm\ref{exist}}, and let a net
$(\boldsymbol{\mu}_s)_{s\in S}\subset\mathcal E^+(\mathbf{A})$ be
strongly bounded and converge vaguely to~$\boldsymbol{\mu}_0$. If,
moreover, $\langle g_i,\mu_s^i\rangle=a_i$ for all $s\in S$ and
$i\in I$, then
\begin{equation}
\langle g_i,\mu_0^i\rangle=a_i\quad\text{for all \ } i\in I.
\label{24}
\end{equation}
\end{lemma}

\begin{proof} Fix $i\in I$. By Corollary~\ref{cor:aux1},
the net $(\mu^i_s)_{s\in S}$ is strongly bounded as well.

Also note that $A_i$ can be written as a countable union of
$\mu_s^i$-integrable sets, where $s\in S$ is given. Indeed, this is
obvious if $A_i$ is a countable union of compact sets; otherwise,
due to condition~(a) of Theorem~\ref{exist}, we have $g_{i,\inf}>0$
and hence $\mu_s^i(A_i)\leqslant a_ig_{i,\inf}^{-1}<\infty$.
Therefore, the concept of local $\mu_s^i$-negligibility and that of
$\mu_s^i$-negligibility coincide. Together with
\cite[Lemma~2.3.1]{F1}, this yields that any proposition holds
$\mu_s^i$-a.e. in~$\mathrm X$ provided it holds n.e. in~$A_i$.

We proceed by establishing (\ref{24}). Of course, this needs to be
done only if the set~$A_i$ is noncompact; then, by condition~(c),
its capacity has to be finite. Hence, by~\cite[Th.~4.1]{F1}, for
every $E\subset A_i$ there exists a measure $\theta_E\in\mathcal
E^+(\,\overline{E}\,)$, called an {\em interior equilibrium
measure\/} associated with~$E$, which admits the properties
\begin{equation}
\theta_E(\mathrm X)=\|\theta_E\|^2=C(E), \label{5}
\end{equation}
\begin{equation}
\kappa(x,\theta_E)\geqslant1\quad\text{n.e. in \ } E. \label{6}
\end{equation}

Also observe that there is no loss of generality in assuming $g_i$
to satisfy~(\ref{growth}) with some $r_i\in(1,\infty)$ and
$\tau_i\in\mathcal E$. Indeed, otherwise, due to condition~(b) of
Theorem~\ref{exist}, $g_i$ has to be bounded from above (say by
$M$), which combined with~(\ref{6}) again gives~(\ref{growth}) for
$\tau_i:=M^{r_i}\,\theta_{A_i}$, $r_i\in(1,\infty)$ being arbitrary.

We treat $A_i$ as a locally compact space with the topology induced
from~$\mathrm X$. Given $E\subset A_i$, let $\chi_E$ denote its
characteristic function and let $E^c:=A_i\setminus E$. Further, let
$\{K_i\}$ be the increasing family of all compact subsets~$K_i$
of~$A_i$. Since $g_i\chi_{K_i}$ is upper semicontinuous on~$A_i$
while $(\mu_s^i)_{s\in S}$ converges to~$\mu_0^i$ vaguely, from
Lemma~\ref{lemma:lower} we get
\[
\langle g_i\chi_{K_i},\mu_0^i\rangle\geqslant\limsup_{s\in
S}\,\langle g_i\chi_{K_i},\mu_s^i\rangle\quad\text{for every \ }
K_i\in\{K_i\}.\] On the other hand, application of Lemma~1.2.2
from~\cite{F1} yields
\[
\langle g_i,\mu_0^i\rangle=\lim_{K_i\in\{K_i\}}\,\langle
g_i\chi_{K_i},\mu_0^i\rangle.\] Combining the last two relations, we
obtain
\[
a_i\geqslant\langle g_i,\mu_0^i\rangle\geqslant\limsup_{(s,K_i)\in
S\times\{K_i\}}\,\langle g_i\chi_{K_i},\mu_s^i\rangle=
a_i-\liminf_{(s,K_i)\in S\times\{K_i\}}\,\langle
g_i\chi_{K^c_i},\mu_s^i\rangle,\] $S\times\{K_i\}$ being the
directed product of the directed sets~$S$ and~$\{K_i\}$. Hence, if
we prove
\begin{equation}
\liminf_{(s,K_i)\in S\times\{K_i\}}\,\langle
g_i\chi_{K^c_i},\mu_s^i\rangle=0, \label{25}
\end{equation}
the desired relation~(\ref{24}) follows.

Consider an interior equilibrium measure $\theta_{K^c_i}$, where
$K_i\in\{K_i\}$ is given. Then application of~Lemma~4.1.1 and
Theorem~4.1 from~\cite{F1} shows that
\[
\|\theta_{K^c_i}-\theta_{\tilde{K}^c_i}\|^2\leqslant
\|\theta_{K^c_i}\|^2-\|\theta_{\tilde{K}^c_i}\|^2\quad\text{provided
\ }K_i\subset\tilde{K}_i.\] Furthermore, it is clear from~(\ref{5})
that the net $\|\theta_{K^c_i}\|$, $K_i\in\{K_i\}$, is bounded and
nonincreasing, and hence fundamental in~$\mathbb R$. The preceding
inequality thus yields that the net
$(\theta_{K^c_i})_{K_i\in\{K_i\}}$ is strongly fundamental
in~$\mathcal E$. Since, clearly, it converges vaguely to zero, the
property~(C$_1$) (see.~Sec.~\ref{sec:2}) implies that zero is also
one of its strong limits and, hence,
\begin{equation}
\lim_{K_i\in\{K_i\}}\,\|\theta_{K^c_i}\|=0. \label{27}
\end{equation}

Write $q_i:=r_i(r_i-1)^{-1}$, where $r_i\in(1,\infty)$ is the number
involved in condition~(\ref{growth}). Combining (\ref{growth}) with
(\ref{6}) shows that the inequality
\[
g_i(x)\chi_{K^c_i}(x)\leqslant\kappa(x,\tau_i)^{1/r_i}
\kappa(x,\theta_{K^c_i})^{1/q_i}\]  subsists n.e.~in~$A_i$ and,
hence, $\mu_s^i$-a.e.~in~$\mathrm X$. Having integrated this
relation with respect to~$\mu_s^i$, we then apply the H\"older and,
subsequently, the Cauchy--Schwarz inequalities to the integrals on
the right. This gives
\begin{eqnarray*}
\langle g_i\chi_{K^c_i},\mu_s^i\rangle&\leqslant&
\Bigl[\int\kappa(x,\tau_i)\,d\mu_s^i(x)\Bigr]^{1/r_i}\,
\Bigl[\int\kappa(x,\theta_{K^c_i})\,d\mu_s^i(x)\Bigr]^{1/q_i}\\
{}
&\leqslant&\|\tau_i\|^{1/r_i}\|\theta_{K^c_i}\|^{1/q_i}\|\mu_s^i\|.
\end{eqnarray*}
Taking limits here along $S\times\{K\}$ and using (\ref{7.1i}) and
(\ref{27}), we obtain~(\ref{25}) as desired.\end{proof}

\begin{corollary}\label{lemma:exist} Under the assumptions of Theorem\/~{\rm\ref{exist}}, we have
\begin{equation}\label{masses}\mathfrak
E^{\boldsymbol{\sigma}}_{\mathbf{f}}(\mathbf{A},\mathbf{a},\mathbf{g})\subset\mathcal
E^+_{\boldsymbol{\sigma}}(\mathbf{A},\mathbf{a},\mathbf{g}).\end{equation}
\end{corollary}

\begin{proof}Fix $\boldsymbol{\gamma}\in\mathfrak
E^{\boldsymbol{\sigma}}_{\mathbf{f}}(\mathbf{A},\mathbf{a},\mathbf{g})$;
then there exists a net $(\boldsymbol{\mu}_s)_{s\in
S}\subset\mathcal
E^+_{\boldsymbol{\sigma},\mathbf{f}}(\mathbf{A},\mathbf{a},\mathbf{g})$
converging to~$\boldsymbol{\gamma}$ strongly and vaguely. Taking a
subnet if necessary, we assume $(\boldsymbol{\mu}_s)_{s\in S}$ to be
strongly bounded. Then, by Lemma~\ref{lemma:exist'}, $\langle
g_i,\gamma^i\rangle=a_i$ for all $i\in I$, which together
with~(\ref{ME}) gives~(\ref{masses}).\end{proof}

\section{Proof of Theorem~\ref{exist}}\label{sec:proof.th.str}

Fix an extremal measure $\boldsymbol{\gamma}\in\mathfrak
E^{\boldsymbol{\sigma}}_{\mathbf{f}}(\mathbf{A},\mathbf{a},\mathbf{g})$~---
it exists by~Lemma~\ref{lemma:WM},~(i), and choose a net
$(\boldsymbol{\mu}_s)_{s\in S}\in\mathbb
M^{\boldsymbol{\sigma}}_{\mathbf{f}}(\mathbf{A},\mathbf{a},\mathbf{g})$
that converges to~$\boldsymbol{\gamma}$ both strongly and vaguely.
We are going to show that $\boldsymbol{\gamma}$ is a minimizer, i.e.
\begin{equation}\label{minimizer0}\boldsymbol{\gamma}\in\mathfrak
S^{\boldsymbol{\sigma}}_{\mathbf{f}}(\mathbf{A},\mathbf{a},\mathbf{g}).\end{equation}

According to Corollary~\ref{lemma:exist}, we have
$\boldsymbol{\gamma}\in\mathcal
E^+_{\boldsymbol{\sigma}}(\mathbf{A},\mathbf{a},\mathbf{g})$. Hence,
(\ref{minimizer0}) will be established once we prove that
$\sum_{i\in I}\,\langle f_i,\gamma^i\rangle$ converges absolutely,
so that
\begin{equation}\label{inclusion}\boldsymbol{\gamma}\in\mathcal
E^+_{\boldsymbol{\sigma},\mathbf{f}}(\mathbf{A},\mathbf{a},\mathbf{g}),\end{equation}
and
\begin{equation}\label{minimizer}G_{\mathbf{f}}(\boldsymbol{\gamma})\leqslant
G^{\boldsymbol{\sigma}}_{\mathbf{f}}(\mathbf{A},\mathbf{a},\mathbf{g}).\end{equation}

To this end, assume Case~I to hold, since otherwise
(\ref{inclusion}) and (\ref{minimizer}) have already been
established by Lemma~\ref{caseii} and Corollary~\ref{IorII},
respectively. Then, in consequence of Lemma~\ref{casei}
(see~(\ref{boundbelow}) with
$\boldsymbol{\mu}=\boldsymbol{\gamma}$),
$\langle\mathbf{f},\boldsymbol{\gamma}\rangle$ is well defined and
\begin{equation}\label{gamma}\langle\mathbf{f},\boldsymbol{\gamma}\rangle>-\infty.\end{equation}
Besides, from the strong and vague convergence
of~$(\boldsymbol{\mu}_s)_{s\in S}$ to~$\boldsymbol{\gamma}$ we
obtain
\begin{equation}G^{\boldsymbol{\sigma}}_{\mathbf{f}}(\mathbf{A},\mathbf{a},\mathbf{g})=\lim_{s\in
S}\,\bigl[\|\boldsymbol{\mu}_s\|^2+2\langle\mathbf{f},\boldsymbol{\mu}_s\rangle\bigr]=
\|\boldsymbol{\gamma}\|^2+2\lim_{s\in
S}\,\langle\mathbf{f},\boldsymbol{\mu}_s\rangle\label{chain1}\end{equation}
(consequently, $\lim_{s\in
S}\,\langle\mathbf{f},\boldsymbol{\mu}_s\rangle$ exists and is
finite) and \begin{align}\notag
\langle\mathbf{f},\boldsymbol{\gamma}\rangle=\sum_{i\in I}\,\langle
f_i,\gamma^i\rangle&\leqslant \sum_{i\in I}\,\liminf_{s\in
S}\,\langle f_i,\mu_s^i\rangle\\{}&\leqslant\lim_{s\in
S}\,\sum_{i\in I}\,\langle f_i,\mu_s^i\rangle=\lim_{s\in
S}\,\langle\mathbf{f},\boldsymbol{\mu}_s\rangle<\infty.\label{chain2}\end{align}
Combining~(\ref{gamma}) and~(\ref{chain2}) proves~(\ref{inclusion}),
while substituting (\ref{chain2}) into~(\ref{chain1})
gives~(\ref{minimizer}), and the required
inclusion~(\ref{minimizer0}) follows.

It has thus been proved that $\mathfrak
E^{\boldsymbol{\sigma}}_{\mathbf{f}}(\mathbf{A},\mathbf{a},\mathbf{g})\subset\mathfrak
S^{\boldsymbol{\sigma}}_{\mathbf{f}}(\mathbf{A},\mathbf{a},\mathbf{g})$.
This inclusion can certainly be inverted, since any minimizer
$\boldsymbol{\lambda}^{\boldsymbol{\sigma}}_{\mathbf{A}}$ can be
thought as an extremal measure generated by the constant net
$(\boldsymbol{\lambda}^{\boldsymbol{\sigma}}_{\mathbf{A}})$. On
account of~(\ref{WM}), we get
\begin{equation*}\label{SWM}\mathfrak S^{\boldsymbol{\sigma}}_{\mathbf{f}}(\mathbf{A},\mathbf{a},\mathbf{g})=
\mathfrak
E^{\boldsymbol{\sigma}}_{\mathbf{f}}(\mathbf{A},\mathbf{a},\mathbf{g})=\mathcal
M^{\boldsymbol{\sigma}}_{\mathbf{f}}(\mathbf{A},\mathbf{a},\mathbf{g}).
\end{equation*}
Therefore Lemma~\ref{lemma:WM} shows that $\mathfrak
S^{\boldsymbol{\sigma}}_{\mathbf{f}}(\mathbf{A},\mathbf{a},\mathbf{g})$
is vaguely compact. The proof is complete.\hfill$\square$

\section{Proof of Theorem~\ref{cor:sigma}}\label{sec:proof.cor:sigma}
It is seen from~(\ref{nonzero1}), (\ref{increas'}) and
Corollary~\ref{lemma:minusfinite} that, under the assumptions of the
theorem,
$G^{\boldsymbol{\sigma}_s}_{\mathbf{f}}(\mathbf{A}_s,\mathbf{a},\mathbf{g})$
increases as $s$ ranges through~$S$ and
\begin{equation*}\label{Gfund}-\infty<\lim_{s\in
S}\,G^{\boldsymbol{\sigma}_s}_{\mathbf{f}}(\mathbf{A}_s,\mathbf{a},\mathbf{g})\leqslant
G^{\boldsymbol{\sigma}}_{\mathbf{f}}(\mathbf{A},\mathbf{a},\mathbf{g})
<\infty.\end{equation*} Besides, in accordance with
Theorem~\ref{exist}, for every $s\geqslant s_0$ there is a minimizer
$\boldsymbol{\lambda}_s:=\boldsymbol{\lambda}^{\boldsymbol{\sigma}_s}_{\mathbf{A}_s}\in\mathfrak
S^{\boldsymbol{\sigma}_s}_{\mathbf{f}}(\mathbf{A}_s,\mathbf{a},\mathbf{g})$.
Therefore, $\lim_{s\in S}\,G_{\mathbf{f}}(\boldsymbol{\lambda}_s)$
exists and
\begin{equation}\label{Gfund2}
-\infty<\lim_{s\in
S}\,G_{\mathbf{f}}(\boldsymbol{\lambda}_s)\leqslant
G^{\boldsymbol{\sigma}}_{\mathbf{f}}(\mathbf{A},\mathbf{a},\mathbf{g})
<\infty.\end{equation} Also observe that, since $\mathbf{A}_s$ and
$\boldsymbol{\sigma}_s$ decrease along~$S$, it is true that
\begin{equation*}\label{belongs}\boldsymbol{\lambda}_{s}\in\mathcal
E^+_{\boldsymbol{\sigma}_\ell,\mathbf{f}}(\mathbf{A}_\ell,\mathbf{a},\mathbf{g})
\quad\text{for all \ }s\geqslant\ell\geqslant s_0.\end{equation*}

We proceed by showing that
\begin{equation}\label{convex}
\|\boldsymbol{\lambda}_{s_2}-
\boldsymbol{\lambda}_{s_1}\|^2\leqslant
G_{\mathbf{f}}(\boldsymbol{\lambda}_{s_2})-G_{\mathbf{f}}(\boldsymbol{\lambda}_{s_1})
\end{equation}
whenever $s_0\leqslant s_1\leqslant s_2$. For every $t\in(0,1]$,
$\boldsymbol{\mu}:=(1-t)\boldsymbol{\lambda}_{s_1}+
t\boldsymbol{\lambda}_{s_2}$ belongs to the class $\mathcal
E^+_{\boldsymbol{\sigma}_{s_1},\mathbf{f}}(\mathbf{A}_{s_1},\mathbf{a},\mathbf{g})$,
and therefore $G_{\mathbf{f}}(\boldsymbol{\mu})\geqslant
G_{\mathbf{f}}(\boldsymbol{\lambda}_{s_1})$. Evaluating the
left-hand side of this inequality and then letting $t\to0$, we get
\[-\|\boldsymbol{\lambda}_{s_1}\|^2+
\kappa(\boldsymbol{\lambda}_{s_1},\boldsymbol{\lambda}_{s_2})-
\langle\mathbf{f},\boldsymbol{\lambda}_{s_1}\rangle+
\langle\mathbf{f},\boldsymbol{\lambda}_{s_2}\rangle\geqslant0,\]
and~(\ref{convex}) follows.

Due to (\ref{Gfund2}), the net
$G_{\mathbf{f}}(\boldsymbol{\lambda}_s)$, $s\in S$, is fundamental
in~$\mathbb R$. When combined with (\ref{convex}), this implies that
$\boldsymbol{\lambda}_{s}$, $s\geqslant\ell\geqslant s_0$, form a
fundamental net in~$\mathcal
E^+_{\boldsymbol{\sigma}_\ell}(\mathbf{A}_\ell,\mathbf{a},\mathbf{g})$.
Hence, by Theorem~\ref{th:strong}, there exists a vague cluster
point~$\boldsymbol{\lambda}$ of $(\boldsymbol{\lambda}_{s})_{s\in
S}$ and the following two assertions hold:
\[\text{(i)}\quad\boldsymbol{\lambda}\in\mathcal
E^+_{\boldsymbol{\sigma}_s}(\mathbf{A}_s,\leqslant\!\mathbf{a},\mathbf{g})\quad\text{
for all\ }s\geqslant s_0;\qquad
\text{(ii)}\quad\boldsymbol{\lambda}_{s}\to\boldsymbol{\lambda}\text{
strongly}.\]

However, Lemma~\ref{lemma:exist'} with $\mathbf{A}_s$ instead of
$\mathbf{A}$ shows that assertion~(i) can be strengthened as
follows: $\boldsymbol{\lambda}\in\mathcal
E^+_{\boldsymbol{\sigma}_s}(\mathbf{A}_s,\mathbf{a},\mathbf{g})$ for
all $s\geqslant s_0$. In turn, this implies that, actually,
$\boldsymbol{\lambda}\in\mathcal
E^+_{\boldsymbol{\sigma}}(\mathbf{A},\mathbf{a},\mathbf{g})$, since
$\boldsymbol{\sigma}_s\to\boldsymbol{\sigma}$ vaguely while
$\mathbf{A}_s\downarrow\mathbf{A}$.

What has already been established yields that the proof of the
theorem will be complete once we show that $\sum_{i\in I}\,\langle
f_i,\lambda^i\rangle$ converges absolutely, so that
\begin{equation}\label{rem1}\boldsymbol{\lambda}\in\mathcal
E^+_{\boldsymbol{\sigma},\mathbf{f}}(\mathbf{A},\mathbf{a},\mathbf{g}),\end{equation}
and
\begin{equation}\label{rem2}
\langle\mathbf{f},\boldsymbol{\lambda}\rangle\leqslant\lim_{s\in
S}\,\langle\mathbf{f},\boldsymbol{\lambda}_{s}\rangle.
\end{equation}
Note that $\lim_{s\in
S}\,\langle\mathbf{f},\boldsymbol{\lambda}_{s}\rangle$ exists and is
finite, which is clear from (\ref{Gfund2}) and (ii).

We can suppose Case~I to hold, since otherwise (\ref{rem1}) is
already known from Lemma~\ref{caseii} while (\ref{rem2}) can be
obtained directly from~(\ref{yu}) and assertion~(ii). Therefore,
by~(\ref{boundbelow}) with $\boldsymbol{\mu}=\boldsymbol{\lambda}$,
$\langle\mathbf{f},\boldsymbol{\lambda}\rangle$ is well defined and
$\langle\mathbf{f},\boldsymbol{\lambda}\rangle>-\infty$. Taking a
subnet if necessary, we can also assume that
$\boldsymbol{\lambda}_{s}\to\boldsymbol{\lambda}$ vaguely. Then,
\begin{equation*}\label{hryu}-\infty<\langle\mathbf{f},\boldsymbol{\lambda}\rangle=
\sum_{i\in I}\,\langle f_i,\lambda^i\rangle\leqslant \sum_{i\in
I}\,\liminf_{s\in S}\,\langle
f_i,\lambda_s^i\rangle\leqslant\lim_{s\in S}\,\sum_{i\in I}\,\langle
f_i,\lambda_s^i\rangle<\infty,\end{equation*} and the required
relations (\ref{rem1}) and~(\ref{rem2}) follow.\hfill$\square$

\section{Proof of Theorem~\ref{cor:cont}}\label{sec:proof.th.cont}
We begin by establishing the relation
\begin{equation}
G^{\boldsymbol{\sigma}}_{\mathbf{f}}(\mathbf{A},\mathbf{a},\mathbf{g})=\lim_{\mathbf{K}\uparrow\mathbf{A}}\,
G^{\boldsymbol{\sigma}_{\mathbf{K}}}_{\mathbf{f}}(\mathbf{K},\mathbf{a},\mathbf{g}).\label{cont}
\end{equation}

For every $\boldsymbol{\mu}\in\mathcal
E^+_{\boldsymbol{\sigma},\mathbf{f}}(\mathbf{A},\mathbf{a},\mathbf{g})$,
consider
$\widehat{\boldsymbol{\mu}}_{\mathbf{K}}=(\hat{\mu}_{\mathbf{K}}^i)_{i\in
I}$ defined by~(\ref{hatlambda''}). Fix an arbitrary $\varepsilon>0$
and choose $\mathbf{K}_0$ so that for all $\mathbf{K}$ that
follow~$\mathbf{K}_0$ inclusion~(\ref{exh1}) holds. This yields
\begin{equation}
G_{\mathbf{f}}(\widehat{\boldsymbol{\mu}}_{\mathbf{K}})\geqslant
G^{(1+\varepsilon)\boldsymbol{\sigma}_{\mathbf{K}}}_{\mathbf{f}}(\mathbf{K},\mathbf{a},\mathbf{g}).\label{www}\end{equation}

We next proceed by showing that
\begin{equation}
G_{\mathbf{f}}(\boldsymbol{\mu})=
\lim_{\mathbf{K}\uparrow\mathbf{A}}\,G_{\mathbf{f}}(\widehat{\boldsymbol{\mu}}_{\mathbf{K}}).\label{4w}\end{equation}
To this end, it can be assumed that $\kappa\geqslant0$; for if not,
then $\mathbf{A}$ must be finite since $\mathrm X$ is compact, and
(\ref{4w}) follows from~(\ref{wnew})--(\ref{wwnew}). Therefore, for
all $\mathbf{K}\geqslant\mathbf{K}_0$ and $i\in I$ we get
\begin{equation}
\|\mu^i_{\mathbf{K}}\|\leqslant\|\mu^i\|\leqslant\|R\boldsymbol{\mu}^++R\boldsymbol{\mu}^-\|,
\label{unif3}
\end{equation}
\begin{equation}
\|\mu^i-\mu^i_{\mathbf{K}}\|<\varepsilon\,i^{-2},\label{uniff}
\end{equation}
the latter being clear from (\ref{unif1new}) because of
$\kappa(\mu^i_{\mathbf{K}},\mu^i-\mu^i_{\mathbf{K}})\geqslant0$.
Also observe that
\begin{equation*}
\begin{split}
\bigl|&\|\boldsymbol{\mu}\|^2-\|\widehat{\boldsymbol{\mu}}_{\mathbf{K}}\|^2\bigr|\leqslant\sum_{i,j\in
I}\,\Bigl|\kappa(\mu^i,\mu^j)- \frac{a_i}{\langle
g_i,\mu_{\mathbf{K}}^i\rangle}\frac{a_j}{\langle
g_j,\mu_{\mathbf{K}}^j\rangle}\,
\kappa(\mu_{\mathbf{K}}^i,\mu_{\mathbf{K}}^j)\Bigr|\\[3pt]
&{}\leqslant\sum_{i,j\in
I}\,\Bigl[\kappa(\mu^i-\mu^i_{\mathbf{K}},\mu^j)+\kappa(\mu^i_{\mathbf{K}},\mu^j-\mu^j_{\mathbf{K}})+
\Bigl(\frac{a_i}{\langle
g_i,\mu_{\mathbf{K}}^i\rangle}\frac{a_j}{\langle
g_j,\mu_{\mathbf{K}}^j\rangle}-1\Bigr)\,\kappa(\mu^i_{\mathbf{K}},\mu^j_{\mathbf{K}})\Bigr].
\end{split}
\end{equation*}
When combined with (\ref{unif2new}), (\ref{unif2'new}),
(\ref{unif3}), and~(\ref{uniff}), this yields
\[
\bigl|G_{\mathbf{f}}(\boldsymbol{\mu})-G_{\mathbf{f}}(\widehat{\boldsymbol{\mu}}_{\mathbf{K}})\bigr|
\leqslant M\varepsilon\quad\text{for all \
}\mathbf{K}\geqslant\mathbf{K}_0,\] where $M$ is finite and
independent of~$\mathbf{K}$, and the required relation~(\ref{4w})
follows.

Substituting (\ref{www}) into~(\ref{4w}), in view of
(\ref{increas'}) and the arbitrary choice of~$\boldsymbol{\mu}$ we
get
\[
G^{\boldsymbol{\sigma}}_{\mathbf{f}}(\mathbf{A},\mathbf{a},\mathbf{g})\geqslant
\lim_{\mathbf{K}\uparrow\mathbf{A}}\,G^{(1+\varepsilon)\boldsymbol{\sigma}_{\mathbf{K}}}_{\mathbf{f}}(\mathbf{K},\mathbf{a},\mathbf{g})
\geqslant
G^{(1+\varepsilon)\boldsymbol{\sigma}}_{\mathbf{f}}(\mathbf{A},\mathbf{a},\mathbf{g}).
\]
Letting $\varepsilon\to0$ and applying Theorem~\ref{sigma} to both
$\mathbf{A}$ and $\mathbf{K}$, we obtain
\[G^{\boldsymbol{\sigma}}_{\mathbf{f}}(\mathbf{A},\mathbf{a},\mathbf{g})=\lim_{\varepsilon\to0}\,
\Bigl[\lim_{\mathbf{K}\uparrow\mathbf{A}}\,
G^{(1+\varepsilon)\boldsymbol{\sigma}_{\mathbf{K}}}_{\mathbf{f}}(\mathbf{K},\mathbf{a},\mathbf{g})\Bigr]=
\lim_{\mathbf{K}\uparrow\mathbf{A}}\,G^{\boldsymbol{\sigma}_{\mathbf{K}}}_{\mathbf{f}}(\mathbf{K},\mathbf{a},\mathbf{g}),\]
which proves (\ref{cont}) as desired.

Fix
$\boldsymbol{\lambda}^{\boldsymbol{\sigma}_{\mathbf{K}}}_{\mathbf{K}}\in\mathfrak
S^{\boldsymbol{\sigma}_{\mathbf{K}}}_{\mathbf{f}}(\mathbf{K},\mathbf{a},\mathbf{g})$,
where $\mathbf{K}\in\{\mathbf{K}\}_{\mathbf{A}}$. As follows
from~(\ref{cont}), the net
$(\boldsymbol{\lambda}^{\boldsymbol{\sigma}_{\mathbf{K}}}_{\mathbf{K}})_{\mathbf{K}\in\{\mathbf{K}\}_{\mathbf{A}}}$
belongs to $\mathbb
M^{\boldsymbol{\sigma}}_{\mathbf{f}}(\mathbf{A},\mathbf{a},\mathbf{g})$
and, hence, it is strongly fundamental.

Further, for every $\mathbf{K}\in\{\mathbf{K}\}_{\mathbf{A}}$ choose
$\beta^*_{\mathbf{K}}\in(1,\infty)$ such that
$(\beta^*_{\mathbf{K}})_{\mathbf{K}\in\{\mathbf{K}\}_{\mathbf{A}}}$
decreases to~$1$ and, for all
$\beta_{\mathbf{K}}\in[1,\beta^*_{\mathbf{K}}]$ and
$\boldsymbol{\lambda}^{\beta_{\mathbf{K}}\boldsymbol{\sigma}_{\mathbf{K}}}_{\mathbf{K}}\in\mathfrak
S^{\beta_{\mathbf{K}}\boldsymbol{\sigma}_{\mathbf{K}}}_{\mathbf{f}}(\mathbf{K},\mathbf{a},\mathbf{g})$,
\begin{equation}\label{last1}\lim_{\mathbf{K}\in\{\mathbf{K}\}_{\mathbf{A}}}\,
\|\boldsymbol{\lambda}^{\beta_{\mathbf{K}}\boldsymbol{\sigma}_{\mathbf{K}}}_{\mathbf{K}}-
\boldsymbol{\lambda}^{\boldsymbol{\sigma}_{\mathbf{K}}}_{\mathbf{K}}\|=0,\end{equation}
\begin{equation}\label{last2}\lim_{\mathbf{K}\in\{\mathbf{K}\}_{\mathbf{A}}}\,
\bigl[G_{\mathbf{f}}(\boldsymbol{\lambda}^{\beta_{\mathbf{K}}\boldsymbol{\sigma}_{\mathbf{K}}}_{\mathbf{K}})-
G_{\mathbf{f}}(\boldsymbol{\lambda}^{\boldsymbol{\sigma}_{\mathbf{K}}}_{\mathbf{K}})\bigr]=0.\end{equation}
The existence of those $\beta^*_{\mathbf{K}}$ follows from
Theorem~\ref{cor:sigma}.

Then, combining (\ref{cont}) and~(\ref{last2}) gives~(\ref{contnew})
as required, while (\ref{last1}) together with the property of
strong fundamentality of
$(\boldsymbol{\lambda}^{\boldsymbol{\sigma}_{\mathbf{K}}}_{\mathbf{K}})_{\mathbf{K}\in\{\mathbf{K}\}_{\mathbf{A}}}$
shows that
$(\boldsymbol{\lambda}^{\beta_{\mathbf{K}}\boldsymbol{\sigma}_{\mathbf{K}}}_{\mathbf{K}})_{\mathbf{K}
\in\{\mathbf{K}\}_{\mathbf{A}}}$ is strongly fundamental as well.
Hence, by Theorem~\ref{th:strong} and Lemma~\ref{lemma:exist'}, the
vague cluster set of
$(\boldsymbol{\lambda}^{\beta_{\mathbf{K}}\boldsymbol{\sigma}_{\mathbf{K}}}_{\mathbf{K}})_{\mathbf{K}
\in\{\mathbf{K}\}_{\mathbf{A}}}$ is nonempty and for every its
element~$\boldsymbol{\lambda}$ the following assertions both hold:
$\boldsymbol{\lambda}\in\mathcal
E^+_{\boldsymbol{\sigma}}(\mathbf{A},\mathbf{a},\mathbf{g})$ and
$\boldsymbol{\lambda}^{\beta_{\mathbf{K}}\boldsymbol{\sigma}_{\mathbf{K}}}_{\mathbf{K}}\to\boldsymbol{\lambda}$
strongly. To complete the proof, it is enough to show that
$\boldsymbol{\lambda}\in\mathfrak
S^{\boldsymbol{\sigma}}_{\mathbf{f}}(\mathbf{A},\mathbf{a},\mathbf{g})$,
but this can be done in the same way as at the end of the proof of
Theorem~\ref{cor:sigma}.\phantom{p}\hfill$\square$

\section{On the sharpness of condition~(c) in Theorem~\ref{exist}}\label{sharp}

Given a closed set $F\subset\mathbb R^n$, for brevity let $\mathfrak
M^1(F)$ denote the collection of all probability measures supported
by~$F$. The examples below illustrate the sharpness of condition~(c)
in Theorem~\ref{exist}.

\subsection{Examples}

\begin{example}\label{121}
In $\mathbb R^n$, $n\geqslant 2$, consider the Riesz kernel
$\kappa_\alpha(x,y):=|x-y|^{\alpha-n}$ of order~$\alpha$,
$\alpha\in(0,2]$, $\alpha<n$, and a condenser
$\mathbf{A}=(A_1,A_2)$, where $I^+=\{1\}$, $I^-=\{2\}$, $A_1$ is
compact, and $C_{\kappa_\alpha}(A_i)>0$ for $i=1,2$. Also consider
the $\alpha$-Green kernel~$g^\alpha_{A^c_2}$ of $A^c_2:=\mathbb
R^n\setminus A_2$, defined by  (see, e.g., \cite[Chap.~4, \S~5]{L})
\[g^\alpha_{A^c_2}(x,y):=\kappa_\alpha(x,\varepsilon_y)-\kappa_\alpha(x,\beta^\alpha_{A_2}\varepsilon_y),\]
where $\varepsilon_y$ is the (unit) Dirac measure at~$y$ and
$\beta^\alpha_{A_2}$ is the operator of Riesz balayage onto~$A_2$.
Further, let $K\subset(A_1\cup A_2)^c$ be a compact set with
$C_{\kappa_\alpha}(K)>0$, and let $\theta$ denote the (unique)
minimizer in the minimal $\alpha$-Green energy problem
\[\inf_{\nu\in\mathfrak M^1(A_1\cup K)}\,g_{A^c_2}^\alpha(\nu,\nu);\]
then it holds true that
\[\|\theta\|^2_{g_{A^c_2}^\alpha}=
\|\theta-\beta^\alpha_{A_2}\theta\|^2_{\kappa_\alpha}=\bigl[C_{g_{A^c_2}^\alpha}(A_1\cup
K)\bigr]^{-1}.\]

Assume, moreover, that $\mathbf{f}$ satisfies Case~II with
$\zeta=\theta_{K}$, where $\theta_{K}$ is the trace of~$\theta$
upon~$K$, and let $a_1=\theta(A_1)$, $a_2=1$, $g_1=g_2=1$. Also
assume, for simplicity, $A^c_2$ to be connected.

\begin{proposition}\label{propex}Under the above notation and requirements,
the following two assertions are equivalent:

\begin{itemize}\item[\rm(i)]
One can choose a strictly increasing sequence of positive numbers
$R_k$, $k\in\mathbb N$, and measures $\omega_k\in\mathfrak
M^1\bigl(A_2\cap\{R_k\leqslant|x|\leqslant R_{k+1}\}\bigr)$ so that
\begin{equation}\label{nonsolv}\mathfrak S_{\mathbf
f}^{\boldsymbol{\sigma}}(\mathbf{A},\mathbf{a},\mathbf{g})=\varnothing,\end{equation}
where $\boldsymbol{\sigma}=(\sigma^1,\sigma^2)$ is a constraint with
the components
\begin{equation}\label{contr}
\sigma^1:=\theta_{A_1}\quad\text{and}\quad
\sigma^2:=\beta^\alpha_{A_2}\theta+\bigl[1-\beta^\alpha_{A_2}\theta(A_2)\bigr]\sum_{k\in\mathbb
N}\,\omega_k.\end{equation}
\item[\rm(ii)] $C_{\kappa_\alpha}(A_2)=\infty$, though $A_2$ is $\alpha$-thin
at $\infty_{\mathbb R^n}$.\end{itemize}\end{proposition}

Recall that a closed set $F\subset\mathbb R^n$ is $\alpha$-{\em thin
at\/} $\infty_{\mathbb R^n}$ if the origin $x=0$ is an
$\alpha$-irregular point for the inverse of~$F$ relative to the unit
sphere (see~\cite{Brelo2}; cf.~also~\cite{Brelo1,L}). See, e.g.,
\cite[Chap.~V]{L} for the notion of $\alpha$-regularity in case
$\alpha\in(0,2)$.\end{example}

\begin{example} With the notation and the requirements of Example~\ref{121},
let $\kappa_2(x,y)$ be the Newtonian kernel $|x-y|^{-1}$ in~$\mathbb
R^3$, and let $A_2$ be a rotational body consisting of all
$x=(x_1,x_2,x_3)\in\mathbb R^3$ such that $q\leqslant x_1<\infty$
and $0\leqslant x_2^2+x_3^2\leqslant\rho(x_1)$, where $q\in\mathbb
R$ and $\rho(x_1)$ approaches~$0$ as $x_1\to\infty$. Consider the
following three cases:
\begin{align}\rho(r)&=r^{-s},\phantom{\exp-}\quad\text{where \
}s\in[0,\infty),\label{ex1}\\
\rho(r)&=\exp(-r^s),\quad\text{where \ }s\in(0,1],\label{ex2}\\
\rho(r)&=\exp(-r^s),\quad\text{where \ }s>1.\label{ex3}\end{align}
As has been shown in~\cite{Z0}, $A_2$ is not $2$-thin
at~$\infty_{\mathbb R^3}$ in case~(\ref{ex1}), has finite
(Newtonian) capacity in case~(\ref{ex3}), and it is $2$-thin
at~$\infty_{\mathbb R^3}$ though has infinite (Newtonian) capacity
in case~(\ref{ex2}). Consequently, assertion~(i) from
Proposition~\ref{propex} on the unsolvability of the corresponding
constrained problem holds in case~(\ref{ex2}), and it fails to hold
in both cases~(\ref{ex1}) and~(\ref{ex3}).\end{example}

\subsection{Proof of Proposition~\ref{propex}}
Our arguments are based on Theorem 4 from~\cite{Z1}, which asserts
that, if $F\subset\mathbb R^n$ is closed, $\nu\geqslant0$ is
concentrated in~$F^c$ and if, for simplicity, $F^c$ is connected,
then
\begin{equation}\label{balayage}\beta^\alpha_{F}\nu(\mathbb R^n)=\nu(\mathbb
R^n)\iff F\text{\ is not $\alpha$-thin at $\infty_{\mathbb
R^n}$.}\end{equation}

Fix an arbitrary $\boldsymbol{\mu}=(\mu^1,\mu^2)\in\mathcal
E^+(\mathbf{A})$. Then, by~(\ref{yusss}),
\[G_{\mathbf{f}}(\boldsymbol{\mu})=
\|R\boldsymbol{\mu}+\theta_K\|_{\kappa_\alpha}^2-\|\theta_K\|_{\kappa_\alpha}^2=
\|\mu^1+\theta_K-\mu^2\|_{\kappa_\alpha}^2-\|\theta_K\|_{\kappa_\alpha}^2.\]
Since, by known facts from the Riesz and $\alpha$-Green potential
theory~\cite{L},
\begin{equation*}\|\mu^1+\theta_K-\mu^2\|_{\kappa_\alpha}^2\geqslant\|\mu^1+\theta_K-
\beta^\alpha_{A_2}(\mu^1+\theta_K)\|_{\kappa_\alpha}^2
=\|\mu^1+\theta_K\|^2_{g^\alpha_{A^c_2}}\geqslant\|\theta\|^2_{g^\alpha_{A^c_2}},
\end{equation*}
we get
\begin{equation}\label{chain}G_{\mathbf{f}}(\boldsymbol{\mu})
\geqslant\|\theta\|^2_{g^\alpha_{A^c_2}}-\|\theta_K\|_{\kappa_\alpha}^2,\end{equation}
where the inequality is actually an equality if and only if
\[\mu^1+\theta_K=\theta\quad\text{(hence, $\mu^1=\theta_{A_1}$) \
and \ }\mu^2=\beta^\alpha_{A_2}\theta.\]

Assume (i) to hold; then necessarily
$C_{\kappa_\alpha}(A_2)=\infty$, for if not, we would arrive at a
contradiction with Theorem~\ref{exist}. To establish~(ii), assume,
on the contrary, that $A_2$ is not $\alpha$-thin at~$\infty_{\mathbb
R^n}$. Then, by~(\ref{balayage}),
\begin{equation}\label{one}\beta^\alpha_{A_2}\theta(A_2)=1\end{equation} and consequently,
by~(\ref{contr}),
\begin{equation}\label{one'}\sigma^1=\theta_{A_1}\quad\text{and}\quad
\sigma^2=\beta^\alpha_{A_2}\theta.\end{equation} It follows from
(\ref{one}) and (\ref{one'}) that $\boldsymbol{\sigma}\in\mathcal
E_{\boldsymbol{\sigma}}^+(\mathbf{A},\mathbf{a},\mathbf{g})$ and
inequality~(\ref{chain}) for $\boldsymbol{\mu}=\boldsymbol{\sigma}$
is actually an equality. Thus, $\boldsymbol{\sigma}\in\mathfrak
S_{\mathbf
f}^{\boldsymbol{\sigma}}(\mathbf{A},\mathbf{a},\mathbf{g})$, which
is impossible by~(\ref{nonsolv}).

Now, assume (ii) to hold. Since $A_2$ is $\alpha$-thin at
$\infty_{\mathbb R^n}$, from~(\ref{balayage}) we get
\begin{equation}\label{c}c:=1-\beta^\alpha_{A_2}\theta(A_2)>0.\end{equation} Choose a strictly
increasing sequence $(R_k)_{k\in\mathbb N}$ with the property
$C_{\kappa_\alpha}\bigl(A_2^{(k)}\bigr)>k$, where
$A_2^{(k)}:=A_2\cap\{R_k\leqslant|x|\leqslant R_{k+1}\}$, which is
possible because of the assumption $C(A_2)=\infty$, and let
$\omega_k$ minimize $\kappa_\alpha(\nu,\nu)$ among all
$\nu\in\mathfrak M^1\bigl(A_2^{(k)}\bigr)$. Then
\[\lim_{k\to\infty}\,\|\omega_k\|_{\kappa_\alpha}^2=0=
\inf_{\nu\in\mathfrak M^1(A_2)}\,\kappa_\alpha(\nu,\nu),\] which
yields by standard arguments that $(\omega_k)_{k\in\mathbb N}$ is a
strong Cauchy sequence in~$\mathcal E_{\kappa_\alpha}(\mathbb R^n)$.
Since $\omega_k\to0$ vaguely, in view of the perfectness
of~$\kappa_\alpha$ we thus get
\begin{equation*}\label{omegastrong}
\omega_k\to0\quad\text{strongly.}\end{equation*}

Consider $\boldsymbol{\mu}_k=(\mu_k^1,\mu_k^2)$, $k\in\mathbb N$,
with $\mu_k^1=\theta_{A_1}$ and
$\mu_k^2=\beta^\alpha_{A_2}\theta+c\omega_k$, where $c$ is given
by~(\ref{c}). Then $\boldsymbol{\mu}_k\in\mathcal
E_{\boldsymbol{\sigma}}^+(\mathbf{A},\mathbf{a},\mathbf{g})$, where
$\boldsymbol{\sigma}$ is defined by~(\ref{contr}), and
\begin{align*}\|\mu_k^1+\theta_K-\mu_k^2\|_{\kappa_\alpha}^2&=
\|\theta-\beta^\alpha_{A_2}\theta-c\omega_k\|^2_{\kappa_\alpha}\\
&{}=
\|\theta\|_{g^\alpha_{A^c_2}}^2+c^2\|\omega_k\|^2_{\kappa_\alpha}-
2c\kappa_\alpha(\omega_k,\theta-\beta^\alpha_{A_2}\theta).\end{align*}
Letting here $k\to\infty$, we then obtain
\[\lim_{k\to\infty}\,G(\boldsymbol{\mu}_k)=\|\theta\|_{g^\alpha_{A^c_2}}^2-
\|\theta_K\|^2_{\kappa_\alpha}\] and so, by~(\ref{chain}),
$(\boldsymbol{\mu}_k)_{k\in\mathbb N}\in\mathbb
M^{\boldsymbol{\sigma}}_{\mathbf{f}}(\mathbf{A},\mathbf{a},\mathbf{g})$.
Since
$\boldsymbol{\mu}_k\to\boldsymbol{\gamma}:=(\theta_{A_1},\beta^\alpha_{A_2}\theta)$
strongly and vaguely, we get $\boldsymbol{\gamma}\in\mathfrak
E^{\boldsymbol{\sigma}}_{\mathbf{f}}(\mathbf{A},\mathbf{a},\mathbf{g})$.
Moreover, in view of the strict positive definiteness of the Riesz
kernel, $\boldsymbol{\gamma}$ is the only element of the class
$\mathfrak
E^{\boldsymbol{\sigma}}_{\mathbf{f}}(\mathbf{A},\mathbf{a},\mathbf{g})$
(see assertion~(ii) of Lemma~\ref{lemma:WM}). However, because
of~(\ref{c}), $\boldsymbol{\gamma}\not\in\mathfrak
S^{\boldsymbol{\sigma}}_{\mathbf{f}}(\mathbf{A},\mathbf{a},\mathbf{g})$.
This proves~(\ref{nonsolv}) and, hence,~(i).\hfill$\square$

\section*{Acknowledgments} The research was supported, in part, by
the "Scholar-in-Residence" program at IPFW, and the author
acknowledges this institution for the support and the excellent
working conditions. The author also thanks Professors
P.\,D.~Dragnev, E.\,B.~Saff, and W.\,L.~Wendland for many valuable
discussions about the content of the paper.


\begin{thebibliography}{00}

\bibitem{B1}
N.~Bourbaki, {\it Topologie g\'en\'erale, Chap.\/}~I--II,
Actualit\'es Sci. Ind., 1142, Paris (1951).

\bibitem{B2}
N.~Bourbaki, {\it Int\'egration, Chap.\/}~I--IV, Actualit\'es Sci.
Ind., 1175, Paris (1952).

\bibitem{Brelo1} M.~Brelot, {\it El\'ements de la th\'eorie classique du potentiel\/},
Les cours Sorbonne, Paris (1961).

\bibitem{Brelo2} M.~Brelot, {\it On topologies and boundaries in
potential theory\/}, Lectures Notes in Math. {\bf 175}, Springer,
Berlin (1971).

\bibitem{Car} H.~Cartan, {\it Th\'eorie du potentiel
Newtonien: \'energie, capacit\'e, suites de potentiels\/}, Bull.
Soc. Math. France {\bf 73} (1945), 74--106.

\bibitem{D1} J.~Deny, {\it Les potentiels d'\'energie finite\/}, Acta Math. {\bf 82} (1950), 107--183.

\bibitem{D2} J.~Deny, {\it Sur la d\'{e}finition de
l'\'{e}nergie en th\'{e}orie du potentiel\/}, Ann. Inst. Fourier
Grenoble {\bf 2} (1950), 83--99.

\bibitem{D} P.\,D.~Dragnev, {\it Constrained energy problems for
logarithmic potentials\/}, Ph.\,D.~Thesis, University of South
Florida, Tampa (1997).

\bibitem{DS} P.\,D.~Dragnev, E.\,B.~Saff, {\it Constrained energy problems with applications
to orthogonal polynomials of a discrete variable\/}, Journal
d'Analyse Math\'{e}matique {\bf 72} (1997), 223–-259.

\bibitem{E1}
R.~Edwards, {\it Cartan's balayage theory for hyperbolic Riemann
surfaces\/}, Ann. Inst. Fourier {\bf 8} (1958), 263--272.

\bibitem{E2}
R.~Edwards, {\it Functional analysis. Theory and applications\/},
Holt. Rinehart and Winston, New York (1965).

\bibitem{F1} B.~Fuglede, {\it On the theory of potentials in
locally compact spaces\/}, Acta Math.~{\bf 103} (1960), 139--215.

\bibitem{F2} B.~Fuglede, {\it Caract\'erisation des noyaux consistants
en th\'eorie du potentiel\/}, Comptes Rendus~{\bf 255} (1962),
241--243.

\bibitem{GR}
A.\,A.~Gonchar, E.\,A.~Rakhmanov, {\it On the equilibrium problem
for vector potentials\/}, Russian Math. Surveys {\bf 40}:4 (1985),
183–-184.

\bibitem{HK}
W.\,K.~Hayman, P.\,B.~Kennedy, {\it Subharmonic functions\/},
Academic Press, London (1976).

\bibitem{K}
J.\,L.~Kelley, {\it General topology\/}, Princeton, New York (1957).

\bibitem{L}
N.\,S.~Landkof, {\it Foundations of modern potential theory\/},
Springer, Berlin (1972).

\bibitem{MS} E.\,H.~Moore, H.\,L.~Smith, {\it A general theory of
limits\/}, Amer. J. Math. {\bf 44} (1922), 102--121.

\bibitem{NS} E.~M.~Nikishin, V.~N.~Sorokin, {\it Rational approximations and orthogonality\/},
Translations of Mathematical Monographs {\bf 44}, Amer. Math. Soc.,
Providence, RI (1991).

\bibitem{O}
M.~Ohtsuka, {\it On potentials in locally compact spaces\/},
J.~Sci.~Hiroshima Univ. Ser.~A-1 {\bf 25} (1961), \mbox{135--352}.

\bibitem{R} E.\,A.~Rakhmanov, {\it Equilibrium
measure and the distribution of zeros of extremal polynomials of a
discrete variable\/}, Sbornik:Mathematics {\bf 187} (1996),
1213--1228.

\bibitem{ST} E.\,B.~Saff, V.~Totik, {\it Logarithmic potentials
with external fields\/}, Sprin\-ger, Berlin (1997).

\bibitem{Z0} N.~Zorii, {\it An extremal problem of the minimum of energy for space
condensers\/}, Ukrain. Math.~J. {\bf 38} (1986), 365--370.

\bibitem{Z1} N.~Zorii, {\it A problem of minimum energy for space
condensers and Riesz kernels\/}, Ukrain. Math.~J. {\bf 41} (1989),
29--36.

\bibitem{Z6} N.~Zorii, {\it Equilibrium problems for potentials
with external fields\/},  Ukrain. Math.~J. {\bf 55} (2003),
1588--1618.

\bibitem{Z7} N.~Zorii, {\it Necessary and sufficient conditions for
the solvability of the Gauss variational problem\/}, Ukrain.
Math.~J. {\bf 57} (2005), 70--99.

\bibitem{Z8a} N.~Zorii, {\it Constrained energy problem with
external fields\/}, Complex Anal. Oper. Theory,
doi:10.1007/s11785-010-0070-9 (2010), 1--11.

\bibitem{Z8} N.~Zorii, {\it Interior capacities of condensers in
locally compact spaces\/}, Potential Anal.,
doi:10.1007/s11118-010-9204-y (2010), 1--41.

\end{thebibliography}
\end{document}